\documentclass[journal]{IEEEtran}

\RequirePackage{etex}
\usepackage[noadjust]{cite}
\usepackage{amsmath,amssymb,amsfonts,amsthm}
\usepackage{bbm}
\usepackage{algorithmic}
\usepackage{graphicx}
\usepackage[tight,hang]{subfigure}

\usepackage{tikz} 
\usepackage{tikz-network}
\usepackage{standalone}
\usepackage{nameref}

\newtheorem{definition}{Definition}

\newtheorem{theorem}{Theorem}

\newtheorem{corollary}{Corollary}
\newtheorem{proposition}{Proposition}
\newtheorem{remark}{Remark}
\newtheorem{example}{Example}
\newtheorem*{excont}{Example \continuation}
\newcommand{\continuation}{??}
\newenvironment{continueexample}[1]
 {\renewcommand{\continuation}{\ref{#1}}\excont[continued]}
 {\endexcont}
\newtheorem{problem}{Problem}
\newtheorem{assumption}{Assumption}

 \global\long\def\G{\mathcal{G}}
 \global\long\def\E{\mathcal{E}}
 \global\long\def\V{\mathcal{V}}

 \global\long\def\aut{\mathrm{Aut}}
 \global\long\def\GG{\Gamma} 
  
\graphicspath{{figures/}}

\newif\ifimportant

\importantfalse

\begin{document}
\title{Forced Symmetric Formation Control}
\author{Daniel Zelazo,~\IEEEmembership{Senior~Member,~IEEE,} Shin-ichi Tanigawa, and Bernd Schulze \thanks{D. Zelazo is with the Faculty of Aerospace Engineering, Technion-Israel Institute of Technology, Haifa 3200003, Israel. 
S. Tanigawa is with the Graduate School of Information Science and Technology at the University of Tokyo. B. Schulze is with the Department of Mathematics and Statistics at Lancaster University. D. Zelazo was supported by the Israel Science Foundation grant no. 453/24 and the Bernard M. Gordon Center for Systems Engineering at the Technion. S.Tanigawa was partially supported by JST PRESTO Grant Number JPMJPR2126.}
}

\maketitle

\begin{abstract}
This work considers the distance constrained formation control problem with an additional constraint requiring that the formation exhibits a specified spatial symmetry.  We employ recent results from the theory of symmetry-forced rigidity to construct an appropriate potential function that leads to a gradient dynamical system driving the agents to the desired formation.  We  show that only $(1+1/|\Gamma|)n$ edges are sufficient to implement the control strategy when there are $n$ agents and the underlying symmetry group is $\Gamma$. This number is considerably smaller than what is typically required from classic rigidity-theory based strategies ($2n-3$ edges). We also provide an augmented control strategy that ensures the agents can converge to a formation with respect to an arbitrary centroid. Numerous numerical examples are  provided to illustrate the main results.
\end{abstract}


\section{Introduction}
\label{sec:introduction}

Many applications for cooperative multi-agent networks require the agents to arrange themselves into some spatial pattern.  This can include alignment of orientations and velocities for flocking behaviors \cite{Igarashi:TAC2009}, or specific formations like spacecraft constellations for sensing \cite{Rosen2000}  or vehicle platoons for autonomous driving \cite{Dai2018}.  
One of the main challenges for the implementation of these applications is to resolve the trade-off between sparsity of information exchange with guarantees on the system performance. In the study of formation control problems, this trade-off is well understood through the lens of \emph{rigidity theory}.  

Rigidity theory studies the solution of a set of geometric constraints on a discrete configuration of points in an Euclidean space.  These constraints can include distance or bearing constraints between pairs of points.  Of interest in rigidity theory is to determine whether the set of polynomial equations representing these constraints (i) has a solution (independence); (ii) has locally isolated solutions (rigidity); or (iii) has exactly one solution in the given space up to isometric motions (global rigidity) \cite{AR, Jackson2007NotesOT}.  

In this paper we will focus on configurations that are constrained by pairwise distances. Such systems are commonly known as (bar-joint) frameworks.  Since checking a  framework for rigidity  is in general very difficult (as it requires solving a system of quadratic equations), a common approach is to check for the linearized (and stronger) notion of rigidity known as ''infinitesimal rigidity" (Section~\ref{sec.rigidity}).

The seminal work from \cite{Krick2009} provided the first formal result showing that (minimal) infinitesimal rigidity (MIR) of the interaction network in a team of integrator agents is required to ensure that  the gradient controller (locally) converges to the correct formation shape.  In the Euclidean space $\mathbb{R}^2$, MIR translates to having $2n-3$ constraints, where $n$ is the number of agents.
Since the work of \cite{Krick2009}, there has been an explosion of research focusing on the formation control problem from a rigidity theory perspective; see the following for an overview \cite{formationbook2, formationbook, Zhao2019}.  Within the controls community, the main concern of these works, however, focuses on understanding the resulting dynamics of the control strategies, with most efforts on the stability and convergence properties of these systems \cite{Sun2016, Dimarogonas2008, Dorfler2010}.  In these works, the assumption of infinitesimal rigidity is taken as a kind of architectural requirement for solving the formation control problem.  That is, there has not been a concerted effort to exploit results from rigidity theory to relax the infinitesimal rigidity assumption. In this paper, we will pursue this using recent advances in the rigidity analysis of symmetric frameworks.

It is well understood in the rigidity community that there are many special configurations that can lead to unexpected flexibility (or rigidity). Symmetry often leads to such special configurations. 
An example in the plane is shown in Fig.~\ref{flex_ex}. 
The  graph  is infinitesimally rigid (in fact, MIR) for almost all realisations as a framework in the plane. See the framework in Fig.~\ref{flex_ex}(a), for example. However, if we place the vertices in the positions shown in Fig.~\ref{flex_ex}(b), then the framework has three reflection symmetries (each of the reflections in the dashed mirror lines maps the framework onto itself) and it is  continuously flexible, as indicated in Fig.~\ref{flex_ex}(c). 
%
%

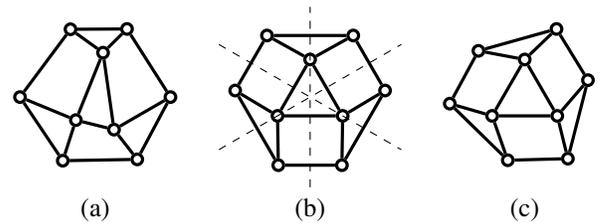
\begin{figure}[!h]
\begin{center}
\begin{tikzpicture}[very thick,scale=1]
\tikzstyle{every node}=[circle, draw=black, fill=white, inner sep=0pt, minimum width=4pt];
\path (80:0.6cm) node (p1)  {} ;
    \path (230:0.4cm) node (p2)  {} ;
    \path (300:0.5cm) node (p3)  {} ;

    \path (65:1cm) node (p4) {} ;
    \path (110:0.96cm) node (p5) {} ;
    \path (180:1cm) node (p6) {} ;
    \path (243:0.95cm) node (p7) {} ;
    \path (304:1cm) node (p8) {} ;
    \path (0:1cm) node (p9) {} ;

    \draw(p1)--(p2);
\draw(p3)--(p2);
\draw(p1)--(p3);

    \draw(p4)--(p5);
\draw(p5)--(p6);
\draw(p6)--(p7);
    \draw(p7)--(p8);
\draw(p8)--(p9);
\draw(p9)--(p4);

    \draw(p1)--(p4);
\draw(p1)--(p5);
\draw(p2)--(p6);

    \draw(p2)--(p7);
\draw(p3)--(p8);
\draw(p9)--(p3);

\node [rectangle, draw=white, fill=white] (b) at (0,-1.5) {(a)};
\end{tikzpicture}
\hspace{0.3cm}
\begin{tikzpicture}[very thick,scale=1]
\tikzstyle{every node}=[circle, draw=black, fill=white, inner sep=0pt, minimum width=4pt];
\path (90:0.5cm) node (p1)  {} ;
    \path (210:0.5cm) node (p2)  {} ;
    \path (330:0.5cm) node (p3)  {} ;

    \path (55:1cm) node (p4) {} ;
    \path (125:1cm) node (p5) {} ;
    \path (175:1cm) node (p6) {} ;
    \path (245:1cm) node (p7) {} ;
    \path (295:1cm) node (p8) {} ;
    \path (5:1cm) node (p9) {} ;

    \draw(p1)--(p2);
\draw(p3)--(p2);
\draw(p1)--(p3);

    \draw(p4)--(p5);
\draw(p5)--(p6);
\draw(p6)--(p7);
    \draw(p7)--(p8);
\draw(p8)--(p9);
\draw(p9)--(p4);

    \draw(p1)--(p4);
\draw(p1)--(p5);
\draw(p2)--(p6);

    \draw(p2)--(p7);
\draw(p3)--(p8);
\draw(p9)--(p3);
\draw[thin, dashed](0,-1.2)--(0,1.2);
\draw[thin, dashed](210:1.4)--(30:1.4);
\draw[thin, dashed](150:1.4)--(-30:1.4);

\node [rectangle, draw=white, fill=white] (b) at (0,-1.5) {(b)};
\end{tikzpicture}
\hspace{0.3cm}
\begin{tikzpicture}[very thick,scale=1]
\tikzstyle{every node}=[circle, draw=black, fill=white, inner sep=0pt, minimum width=4pt];
\path (90:0.5cm) node (p1)  {} ;
    \path (210:0.5cm) node (p2)  {} ;
    \path (330:0.5cm) node (p3)  {} ;

    \path (65:1cm) node (p4) {} ;
    \path (135:0.87cm) node (p5) {} ;
    \path (185:1cm) node (p6) {} ;
    \path (255:0.87cm) node (p7) {} ;
    \path (305:1cm) node (p8) {} ;
    \path (15:0.87cm) node (p9) {} ;

    \draw(p1)--(p2);
\draw(p3)--(p2);
\draw(p1)--(p3);

    \draw(p4)--(p5);
\draw(p5)--(p6);
\draw(p6)--(p7);
    \draw(p7)--(p8);
\draw(p8)--(p9);
\draw(p9)--(p4);

    \draw(p1)--(p4);
\draw(p1)--(p5);
\draw(p2)--(p6);

    \draw(p2)--(p7);
\draw(p3)--(p8);
\draw(p9)--(p3);

\node [rectangle, draw=white, fill=white] (b) at (0,-1.5) {(c)};
\end{tikzpicture}
                  \end{center}
\vspace{-0.3cm}
\caption{
The  framework in (a) is  infinitesimally rigid, whereas the  framework with dihedral symmetry in (b) is  flexible, as shown in (c). }
\label{flex_ex}
\end{figure}

Note that the motion of the framework in (b) destroys all the reflection symmetries.  So while the framework is  flexible, it is still ''forced-symmetric  rigid", in the sense that it does not have a  non-trivial  motion that preserves the original symmetry of the framework. We will make this precise in Section~\ref{sec:symmetry_fwks}. 


In our previous work \cite{Zelazo_IFAC2023}, we proposed a ``symmetry-attaining" formation controller by augmenting the classic gradient-based formation control with another control term that forces the agents into a symmetric position. Under some mild assumptions of the underlying graph, we showed that this approach does not require more communication or sensing than the normal formation controller and can guarantee local convergence to the desired symmetry-constrained formation. Moreover, by means of simulation we demonstrated that the problem can be solved even for flexible frameworks requiring less than $2n-3$ edges as is typically assumed for the formation control problem in MIR frameworks.


This initial work, however, did not leverage the full potential of recent results characterizing symmetry-forced rigidity.  In the present article, we extend the results of \cite{Zelazo_IFAC2023} in the following 
ways. We develop a concrete mathematical foundation for solving the multi-agent formation control problem under symmetry constraints. The theory provides the symmetric counterpart to the ordinary rigidity-based formation control using the modern language of graph rigidity, and in particular that of forced-symmetric rigidity theory. Notably, we show  that $(1+1/|\Gamma|)n$ edges are sufficient when the underlying symmetry group is $\Gamma$. This is significantly smaller than the bound for infinitesimal rigidity.  \textcolor{black}{This improved bound has direct implications for practical implementation as it leads to a reduction of energy consumption and communication bandwidth.  Moreover, this approach can be used to augment any multi-agent coordination problem with explicit symmetry constraints, providing a new conceptual solution to these problems.}
In this direction, we first present a detailed overview of results from the rigidity theory of symmetric frameworks.  The central objects in this study are the \emph{quotient graphs}, providing a graph-theoretic characterization of the vertex and edge orbits induced by a given symmetry group, and the \emph{orbit rigidity matrix}, which can be thought of as the rigidity matrix associated to the quotient graph of a framework. The orbit rigidity matrix is then used to define an \emph{orbit rigidity formation potential} for solving the forced-symmetric formation control problem.  We then provide a stability and convergence analysis of this control law.  Finally, we propose a consensus-augmented version of the forced-symmetric formation control problem that ensures the formation converges to a point other than a globally defined origin.  Numerous examples are provided throughout to illustrate the main concepts and results.





This work is organised as follows.  Section \ref{sec.prelim} presents an overview of geometric rigidity theory and formation control. Section \ref{sec.symmetry} provides a detailed background on notions of symmetry for graphs and frameworks. Section \ref{sec.symform} presents the main results of the paper with different formation control strategies that exploit symmetry properties of the framework.  

\section{Preliminaries from geometric rigidity\\ and Formation Control}\label{sec.prelim}

This section provides an overview of basic concepts from geometric rigidity theory and the formation control literature.  

\subsection{Rigidity Theory}\label{sec.rigidity}
A \emph{framework}  in $\mathbb{R}^d$  is defined to be a pair $(\G,p)$ consisting of a finite simple graph $\G=(\V,\E)$ and a map $p:\V\to \mathbb{R}^d$. It is natural to also consider $p$ as a point in $\mathbb{R}^{d|\V|}$, and  we refer to $p$ as a \emph{configuration} of $|\V|$ points in $\mathbb{R}^d$.   Frameworks are the fundamental objects of study in geometric rigidity theory, where they are considered as mathematical models of physical structures consisting of fixed-length bars (corresponding to the edges of $\G$) that are connected by joints (corresponding to the vertices of $\G$) that allow rotation in any direction. A framework $(\G,p)$ is \emph{rigid} if the only edge-length-preserving continuous motions of the
vertices arise from isometries of $\mathbb{R}^d$, and \emph{flexible} otherwise. The rigidity and flexibility analysis of frameworks is a well-developed theory with a rich history and many practical applications (see, e.g. \cite{bernd2017,ConnellyGuest,Wlong}). In particular, it has recently  found important applications in the formation control of multi-robot systems \cite{formationbook,formationbook2}. 

Since for $d\geq 2$, it is NP-hard to determine if a given framework is rigid, a common  approach to study the rigidity of  frameworks is to linearise the problem by differentiating the length constraints on the edges. This leads to the notion of infinitesimal (or equivalently, static) rigidity.
An \emph{infinitesimal motion} of a framework $(\G,p)$  in $\mathbb{R}^d$
is an assignment of velocity vectors, one to each vertex, $u: \V\to \mathbb{R}^{d}$, such that
$\langle p_i-p_j, u_i-u_j\rangle =0 \quad\textrm{ for all } ij \in  \E\textrm{,}$
where $p_i=p(i)$ and $u_i=u(i)$ for each $i \in \V$. An infinitesimal motion $u$ of $(\G,p)$ is \emph{trivial}
if there exists a skew-symmetric matrix $S$
and a vector $c$ such that $u_i=S p_i+c$ for all $i\in \V$, and $(\G,p)$ is \emph{infinitesimally rigid} if every infinitesimal motion of $(\G,p)$ is trivial, and \emph{infinitesimally flexible} otherwise.

The $|\E| \times d|\V|$ matrix corresponding to the linear system above is called the \emph{rigidity matrix} of $(\G,p)$, denoted as $R(\G,p)$. The row of $R(\G,p)$ corresponding to the edge $ij\in\E$ is of the form 
$\left(\begin{array}{ccccc}   0  \cdots  0 & (p_{i}-p_{j})^T & 0  \cdots  0 & (p_{j}-p_{i})^T &  0   \cdots  0 \end{array} \right).
$
We also employ the useful algebraic representation of the rigidity matrix,
\begin{align}\label{rigidity_matrix}
R(\G,p) = \mathrm{diag}\{(p_i-p_j)^T\}_{ij\in \E}(E^T\otimes I_d),
\end{align}
where $E$ is the  $|\V|\times |\E|$ incidence matrix (using an arbitary orientation of the edges) with $E_{ij}=1$ if the edge $e_j$ leaves vertex $i$, $E_{ij}=-1$ if the edge $e_j$ enters the vertex $i$, and $E_{ij}=0$ otherwise.
So the kernel of  $R(\G,p)$ is the space of all infinitesimal motions of $(\G,p)$, and it is well known that  $(\G,p)$  is infinitesimally rigid if and only if the rank of  $R(\G,p)$ is $d|\V|-\binom{d+1}{2}$, provided that the points $p_i$ affinely span all of $\mathbb{R}^d$ \cite{Wlong}. 

A \emph{self-stress} of a framework $(\G,p)$ is a function $\omega: \E\to \mathbb{R}$ so that the following equation is satisfied at every vertex $i$:
\begin{equation*}\label{eq:equilibrium_stress}
    \sum\limits_{\{j: ij\in E\}} \omega_{ij} (p_i - p_j) = 0.
\end{equation*}
Equivalently, $\omega\in \mathbb{R}^{|\E|}$ is a self-stress if and only if $\omega$ is an element of the cokernel of $R(\G,p)$, i.e. $\omega^TR(\G,p)=0$. If $(\G,p)$ has no non-zero self-stress, then it is called \emph{independent}. Moroever, an infinitesimally rigid and independent framework is called \emph{isostatic}. Isostatic frameworks are also called \emph{minimally infinitesimally rigid}, as the removal of any edge yields an infinitesimally flexible framework.

While an infinitesimally rigid framework is always rigid, the converse does not hold in general. Asimov and Roth, however, showed that for `generic' configurations $p$ (i.e., an open dense subset of configurations), infinitesimal rigidity is equivalent to rigidity \cite{AR}. The graphs that yield rigid frameworks for generic configurations in the plane have been characterized by Pollaczek-Geiringer \cite{gei} and Laman \cite{laman}, and these conditions can be checked in polynomial time. It remains a key open problem to find an efficient characterization of generically rigid graphs in higher dimensions. 

\subsection{Formation Control}\label{sec.formationcontrol}

We review the now well-studied distance-constrained formation control problem \cite{formationbook}.  Consider a network of $n=|\V|$ agents described by integrator dynamics,
\begin{align}\label{integrator}
    \dot{p}_i(t) =  u_i(t),
\end{align}
where $p_i(t) \in \mathbb{R}^d$ is the position of agent $i$, and $u_i(t)\in \mathbb{R}^d$ is the control. As in the development of rigidity theory in \S \ref{sec.rigidity}, the configuration of the network is the stack of the agent positions, $p(t) = \begin{bmatrix}p_1^T(t) &\cdots & p_n^T(t)\end{bmatrix}^T$ (similary defined for $u(t)$). The agents are tasked with attaining a spatial formation using only measurements and/or communication with neighboring agents, as defined by a graph $\G=(\V,\E)$.  The formation is specified by a set of desired inter-agent distances, $\mathrm{\bf d}_{ij}$, for each edge $ij \in \E$, and we denote $\mathrm{\bf d}$ as the stack of all desired distances.  

It is well-known that a gradient-based control strategy can (locally) solve the formation control problem.  In this direction, we define the \emph{formation potential} function,
\begin{align}\label{formation_potential}
F_f(p(t)) =\frac{1}{4} \sum_{ij \in \E}\left(\|p_i(t)-p_j(t)\|^2-\mathrm{\bf d}_{ij}^2\right)^2.
\end{align} 
Then the gradient controller $u(t) = -\nabla F_f(p(t))$ solves the formation control problem.  That is, the closed-loop system
\begin{align} \label{formation_control}
    \dot{p}(t) & = -R(\G,p(t))^T\left(R(\G,p(t))p(t)-\mathrm{\bf d}^2 \right),
\end{align}
satisfies $\underset{t\to \infty}{\lim} \|p_i(t)-p_j(t)\| = \mathrm{\bf d}_{ij}$ for all $ij\in\E$.  Here, $R(\G,p(t))$ is the rigidity matrix defined in \S\ref{sec.rigidity}.  The interested reader may refer to \cite{formationbook2, Sun2016} for more details. 

\section{Symmetry in graphs and frameworks}\label{sec.symmetry}

The main focus of this work is to exploit  notions from the rigidity theory of symmetric frameworks  to solve the formation control problem.  In this section, we provide an overview of  symmetric frameworks and their (infinitesimal) rigidity.  

\subsection{Symmetry in graphs}

Symmetry in objects can be described mathematically via the fundamental algebraic notion of a group (see e.g. \cite{Dummit2004}). 
\begin{definition}
A \emph{group} is defined to be a set, $\Gamma$, together with an operation $\circ$, such that for any two elements $a,b\in \Gamma$, the composition $a \circ b$ is also in $\Gamma$.
The operation $\circ$ satisfies the associativity law. Moreover, each group has a special element $\mathrm{id}$, called the
\emph{identity element}, such that for any element $a \in \Gamma$, $a \circ \mathrm{id}=\mathrm{id} \circ a= a$. Each element $a$ of $\Gamma$ also has an \emph{inverse} $a^{-1}$ in $\Gamma$ such
that $a\circ a^{-1} = a^{-1} \circ a = \mathrm{id}$. The number of elements in a group is called the \emph{order} of the group. A subset $B$ of $\Gamma$ that also forms a group under $\circ$ is called a \emph{subgroup} of $\Gamma$.\end{definition} 

The combinatorial symmetries of a finite simple graph $\G=(\V,\E)$ are described by its group of automorphisms. An automorphism of $\G$ can be loosely understood as a permutation of $\V$ that maps adjacent vertices to adjacent vertices, and non-adjacent vertices to nonadjacent vertices, and hence preserves all structural properties of $\G$. 

\begin{definition}
An \emph{automorphism} of the graph $\G=(\V,\E)$ is a permutation $\psi : \V \to \V$ of its vertex set such that $\psi(v)\psi(u) \in \E \Leftrightarrow vu \in \E.$
\end{definition}

 It is clear, then, that the identity permutation, denoted $\mathrm{id}$, is an automorphism of any graph, and for an automorphism $\psi$, $\psi^{-1}$ is also an automorphism. Therefore, the set of all automorphisms of $\G$ forms a group under composition of maps. This group is called the \emph{automorphism group} of $\G$ and is denoted by $\aut(\G)$.

A common way to represent a permutation for an automorphism is by a two-row array.  For a graph $\G$ with $|\V|=n$ vertices, one can write the automorphism $\psi$ as
$$\psi= \left(\begin{array}{cccc} 1 & 2 & \cdots & n \\ \psi(1) & \psi(2) & \cdots &\psi(n) \end{array} \right).$$
Equivalently, one can express every permutation more compactly as a composition of disjoint cycles of the permutation.  A cycle is a successive action of the permutation that sends a vertex back to itself, i.e., $i\to \psi(i) \to \psi(\psi(i)) \to \cdots \to \psi^k(i)=i$, where $\psi^k = \underbrace{\psi \circ \cdots \circ \psi}_{k \text{ times }}$.  {Such a cycle is compactly written using the \emph{cycle notation}, denoted by $(i\,\psi(i)\,\cdots \psi^{k-1}(i))$.} The integer $k$ is the \emph{length} of the cycle.

\begin{definition}\label{def:free}
A graph $\G$ is \emph{$\Gamma$-symmetric} for any sub-group $\Gamma \subseteq \aut(\G)$.  The symmetry is \emph{free} if $\gamma(i) \neq i$ for all $i\in \V$ and non-identity $\gamma \in \GG$.
\end{definition}

A key structural property of a $\Gamma$-symmetric graph $\G$ is its sets of vertex and edge orbits under $\Gamma$. Loosely speaking, the orbit of a vertex $i$ (or edge $e$) of $\G$ under $\Gamma$ is the set of vertices (edges, resp.) of $\G$ that $i$ ($e$, resp.) can be mapped to by elements in $\Gamma$.

\begin{definition} \label{def:orb}
For a $\GG$-symmetric graph $\G=(\V,\E)$ and vertex $i\in \V$, the set $\GG_i=\{\gamma(i) \,|\, \gamma \in \GG\}$ is called the \emph{vertex orbit} of $i$.  Similarly, for an edge $e=ij \in \E$, the set $\GG_e = \left\{\gamma(i)\gamma(j) \,|\, \gamma \in \GG\right\}$ is termed the \emph{edge orbit} of $e$.
\end{definition}
The size of the vertex orbits depends, of course, on the group $\Gamma$. 
Since all nodes in a given orbit are somehow equivalent under a group action, we often consider \emph{representative vertices} from each vertex orbit.  We denote by $\mathcal V_0$ the set of representative vertices for each orbit, such that $|\mathcal V_0|$ are the number of vertex orbits, and $i\in \mathcal{V}_0$ means that vertex $i$ is only in one vertex orbit.  Similarly, we denote by $\mathcal{E}_0$ the set of representative edges from each edge orbit.

For a $\Gamma$-symmetric graph $\G$ and a vertex orbit $\Gamma_i$ of $\G$ under $\Gamma$, it follows immediately from the definition of $\Gamma_i$ that for every vertex $j$ in $\Gamma_i$ there is a  $\gamma_j \in \Gamma$ such that $\gamma_j(j)=i$. Similarly, for any two vertices $u,v \in \Gamma_i$ there is an automorphism $\gamma_{uv} \in \Gamma$ such that $\gamma_{uv}(u)=v$. With this notation we then have $\gamma_j = \gamma_{ji}$ when $i$ is the representative vertex of $\Gamma_i$. 


\begin{example} \label{ex:aut}
Figure~\ref{fig:square_reflections} shows the cycle graph on $4$ nodes, $C_4$. We will identify all the automorphisms of $\aut(C_4)$.  First, consider a clock-wise rotation by $90^\circ$ of $C_4$ as drawn in the figure.  This gives the automorphism
$$\psi_1 =\left(\begin{array}{cccc} 1 & 2 & 3 & 4 \\ 2 & 4 & 1 &3 \end{array} \right).$$ The cycle notation is $\psi_1=(1\,2\,4\,3)$.
With $\psi_1$, we also have $\psi_2=\psi_1^2=\psi_1\circ \psi_1, \psi_3=\psi_1^3$ and $\psi_1^4=\mathrm{id}$ in $\aut(\G)$, where $\psi_2=(1\,4)(2\, 3)$ and $\psi_3=(1\,3\,4\,2)$ may be interpreted geometrically as rotations by $180^\circ$ and $270^\circ$.  Additional permutations can be found by considering reflections.  
\begin{figure}[!h]
\begin{center}
        \begin{tikzpicture}[very thick,scale=.75]
\tikzstyle{every node}=[circle, draw=black, fill=white, inner sep=0pt, minimum width=5pt];
  \draw [dashed, thin, blue] (-1.5,0) -- (1.5,0);  
        \draw [dashed, thin, red] (0,-1.5) -- (0,1.5);
       \draw [dashed, thin, brown] (-1.5,-1.5) -- (1.5,1.5);
       \draw [dashed, thin, green] (1.5,-1.5) -- (-1.5,1.5);
    \path (-0.8,-0.8) node (p3) [label = below left: $3$] {} ;
    \path (0.8,-0.8) node (p4) [label = below right: $4$] {} ;
    \path (0.8,0.8) node (p2) [label = above right: $2$] {} ;
     \path (-0.8,0.8) node (p1) [label = above left: $1$] {} ;
      \draw (p1) -- (p3);
    \draw (p3) -- (p4);
    \draw (p2) -- (p4);
    \draw (p2) -- (p1);
      
        \end{tikzpicture}
    \end{center}
\vspace{-0.3cm}
\caption{The cycle graph $C_4$ has $8$ automorphisms in $\aut(\G)$.}\label{fig:square_reflections}
\end{figure}
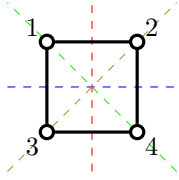
Figure \ref{fig:square_reflections} shows 4 reflection symmetries.  Consider first the reflection about the vertical red line, giving the permutation (in cycle notation) $\psi_4=(1 \, 2)(3 \, 4)$.  Similarly, the horizontal reflection (blue line) yields $\psi_5=(1\,3)(2\,4)$, the diagonal reflection (green line) gives $\psi_6=(1)(2,\, 3)(4)$, and the brown line reflection $\psi_7=(1\, 4)(2)(3)$.  Thus, we have that $\aut(C_4)=\{\mathrm{id},\psi_1,\ldots,\psi_7\}$ has 8 automorphisms. As an abstract group, it is the dihedral group $D_8$ of order $8$ \cite{Dummit2004}. Note that any vertex can be mapped to any other under the automorphisms in $\aut(C_4)$ and hence $C_4$ has only one vertex orbit (and only one edge orbit) under $\aut(C_4)$. If, however, we considered $C_4$ as a $\Gamma$-symmetric graph, where $\Gamma=\{\mathrm{id},\psi_4\}\subset \aut(C_4)$, then $C_4$ has two vertex orbits, namely $\{1,2\}$ and $\{3,4\}$, and three edge orbits, namely $\{12\}, \{34\}$ and $\{13,24\}$. 
\end{example}

\subsection{Symmetry in frameworks}\label{sec:symmetry_fwks}

Having defined notions of symmetries for graphs, we now consider symmetry of frameworks \cite{Bernd2017sym}.

\begin{definition}\label{def:tau-gamma}
Let  $\G$ be a  $\Gamma$-symmetric graph, and let $\Gamma$ be represented as a \emph{point group}, i.e., a subgroup of the orthogonal group $O(\mathbb{R}^d)$,  via a homomorphism $\tau:\Gamma\rightarrow O(\mathbb{R}^d)$. In other words, $\tau$ assigns an orthogonal matrix (describing an isometry of $\mathbb{R}^d$ such as a rotation or reflection) to each element of $\Gamma$. Then a framework $(\G,p)$ in $\mathbb{R}^d$ is called \emph{$\tau(\Gamma)$-symmetric} if 
\begin{equation}\label{eq:symfwk}
\tau(\gamma) (p_i)=p_{\gamma (i)} \quad \textrm{for all } \gamma\in \Gamma \textrm{ and all } i\in \V.
\end{equation}
\end{definition}

Given  a $\tau(\Gamma)$-symmetric framework $(\G,p)$ and a vertex orbit $\Gamma_i$, for every $j\in \Gamma_i$ there is a  $\gamma_j \in \Gamma$ such that  $\tau(\gamma_j)p_j = p_i$. More generally,  for any two vertices $u,v\in \Gamma_i$ there exists $\gamma_{uv} \in \Gamma$ such that $\tau(\gamma_{uv})p_u=p_v$.  

We use the standard Schoenflies notation for point groups in this paper \cite{alt94,atk70}. The only possible point groups in dimension $2$ are the reflection group $\mathcal{C}_s$ (consisting of the identity and a single reflection about an axis $\sigma$), the rotational groups $\mathcal{C}_n$ of order $n$, where $n\geq 1$ (generated  by a rotation $c_n$ about the origin in counter-clockwise direction by  an angle of $2\pi/n$), and the dihedral groups $\mathcal{C}_{nv}$ of order $2n$, where $n\geq 2$ (generated by a reflection $\sigma$ and a rotation $c_n$). If we think of the graph drawing in Figure~\ref{fig:square_reflections} as a framework in the plane, for example, then this framework is $\mathcal{C}_{4v}$-symmetric. 

If a framework $(\G,p)$ is $\tau(\Gamma)$-symmetric, then the configuration $p$ is in a special geometric position that may no longer be `generic.' Thus, symmetry can lead to unexpected  flexibility (as well as unexpected rigidity), as in Fig.~\ref{flex_ex}. Since symmetry is very common in both natural and man-made structures, the rigidity and flexibility analysis of symmetric frameworks has grown into a major research area over the last two decades; see \cite{Bernd2017sym}  for a summary of this work.

A fundamental result  -- based on group representation theory -- is that for a $\tau(\Gamma)$-symmetric framework $(\G,p)$, there are suitable symmetry-adapted bases of $\mathbb{R}^{|\E|}$ and $\mathbb{R}^{|d\V|}$ that transform the rigidity matrix of $(\G,p)$ into a block-decomposed form \cite{kangwai2000,op,bernd2010}. This block-decomposition of $R(\G,p)$ can be used to break up the infinitesimal rigidity analysis of $(\G,p)$ into a number of independent sub-problems, one for each block matrix. A number of results concerning the infinitesimal rigidity of symmetric frameworks have been obtained via this approach (see e.g. \cite{SchTan15,ike}).

Another major research area is to study the rigidity of symmetric frameworks under the additional constraint that any motion must  preserve the original symmetry of the framework. If all symmetry-preserving motions of a $\tau(\Gamma)$-symmetric framework $(\G,p)$ are trivial, then $(\G,p)$ is called \emph{forced $\tau(\Gamma)$-symmetric rigid}.  This is a weaker notion than rigidity, because a forced $\tau(\Gamma)$-symmetric framework may still have non-trivial motions that destroy the original symmetry of the framework. By differentiating a symmetry-preserving continuous motion of $(\G,p)$, we obtain a $\tau(\Gamma)$-symmetric infinitesimal motion, that is, an infinitesimal motion whose velocity assignments to the points of $(\G,p)$  exhibit exactly the same symmetry as the configuration $p$. 

\begin{definition}
An infinitesimal motion $u$ of a $\tau(\Gamma)$-symmetric framework $(\G,p)$ is  \emph{$\tau(\Gamma)$-symmetric}  if  \begin{equation}\label{eq:symmot}
\tau(\gamma)(u_i)=u_{\gamma( i)} \quad\textrm{ for all } \gamma\in \Gamma \textrm{ and all } i\in  \V.
\end{equation}
We say that $(\G,p)$ is \emph{$\tau(\Gamma)$-symmetric infinitesimally rigid} if every $\tau(\Gamma)$-symmetric infinitesimal motion is trivial.

A self-stress $\omega$ of $(\G,p)$ is  \emph{$\tau(\Gamma)$-symmetric}  if 
\begin{equation}\label{eq:symss} \omega_{\gamma(e)}=\omega_e\quad\textrm{ for all } \gamma\in \Gamma \textrm{ and all } e\in  \E.
\end{equation}
The framework $(\G,p)$ is  \emph{$\tau(\Gamma)$-symmetric independent} if it has no non-zero $\tau(\Gamma)$-symmetric self-stress. Further, $(\G,p)$ is \emph{$\tau(\Gamma)$-symmetric
isostatic}  if it is  $\tau(\Gamma)$-symmetric infinitesimally rigid and $\tau(\Gamma)$-symmetric independent.
\end{definition}

Note that $(\G,p)$ is $\tau(\Gamma)$-symmetric isostatic if it is  minimally $\tau(\Gamma)$-symmetric infinitesimally rigid, in the sense that the removal of any edge orbit $\Gamma_e$ from $\G$ yields a $\tau(\Gamma)$-symmetric infinitesimally flexible framework \cite{sw2011}.
See Fig.~\ref{symfwks}(a) and (b) for some examples of minimally $\tau(\Gamma)$-symmetric infinitesimally rigid frameworks in the plane. 

\begin{example}
Consider the frameworks in Figure~\ref{symfwks}. They are all infinitesimally (in fact, continuously) flexible. The framework in (a) is (minimally) $\mathcal{C}_4$-symmetric infinitesimally rigid, where $\mathcal{C}_4$ is the rotational point group of order $4$, as the only $\mathcal{C}_4$-symmetric infinitesimal motions are trivial rotations. (Note that $\mathcal{C}_4$ symmetry implies the larger $\mathcal{C}_{4v}$ symmetry in this example and the framework is also (minimally) $\mathcal{C}_{4v}$-symmetric infinitesimally rigid.) 

The frameworks in (b)  and (c) are  $\mathcal{C}_s$- and $\mathcal{C}_2$-symmetric, respectively. The one in (b) is (minimally) $\mathcal{C}_s$-symmetric infinitesimally rigid, whereas the one in (c) is $\mathcal{C}_2$-symmetric infinitesimally flexible. In fact, it has a continuous motion that preserves the half-turn symmetry.
\end{example}
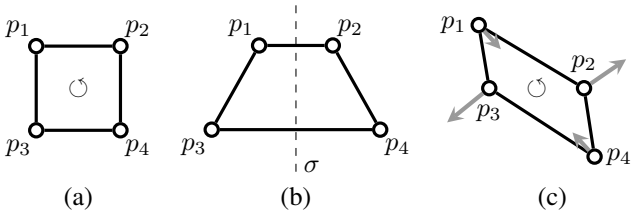
\begin{figure}[htp]
\begin{center}
        \begin{tikzpicture}[very thick,scale=0.7]
\tikzstyle{every node}=[circle, draw=black, fill=white, inner sep=0pt, minimum width=5pt];
    \path (-0.8,-0.8) node (p3) [label = below left: $p_3$] {} ;
    \path (0.8,-0.8) node (p4) [label = below right: $p_4$] {} ;
    \path (0.8,0.8) node (p2) [label = above right: $p_2$] {} ;
     \path (-0.8,0.8) node (p1) [label = above left: $p_1$] {} ;
      \draw (p1) -- (p3);
    \draw (p3) -- (p4);
    \draw (p2) -- (p4);
    \draw (p2) -- (p1);
    \node [draw=white, fill=white] (b) at (0,0) { $\circlearrowleft$};
          \node [draw=white, fill=white] (b) at (0,-2.1) {(a)};
        \end{tikzpicture}
        \hspace{0.1cm}
        \begin{tikzpicture}[very thick,scale=0.7]
\tikzstyle{every node}=[circle, draw=black, fill=white, inner sep=0pt, minimum width=5pt];
    \path (-0.7,0.8) node (p1) [label = above left: $p_1$] {} ;
    \path (0.7,0.8) node (p2) [label = above right: $p_2$] {} ;
    \path (-1.6,-0.8) node (p3) [label = below left: $p_3$] {} ;
     \path (1.6,-0.8) node (p4) [label = below right: $p_4$] {} ;
      \draw (p1) -- (p3);
    \draw (p3) -- (p4);
    \draw (p2) -- (p4);
    \draw (p2) -- (p1);
     \draw [dashed, thin] (0,-1.6) -- (0,1.6);
     \node [draw=white, fill=white] (b) at (0.3,-1.5) {$\sigma$};
      \node [draw=white, fill=white] (b) at (0,-2.1) {(b)};
        \end{tikzpicture}
        \hspace{0.1cm}
\begin{tikzpicture}[very thick,scale=0.7]
\tikzstyle{every node}=[circle, draw=black, fill=white, inner sep=0pt, minimum width=5pt];
    \path (0.1,1.2) node (p1) [label = left: $p_{1}$] {} ;
    \path (2.1,0) node (p4) [label = above: $p_{2}$]{} ;
    \path (2.3,-1.3) node (p3) [label = right: $p_{4}$] {} ;
     \path (0.3,0) node (p2) [label = below: $p_{3}$] {} ;
        \draw (p1) -- (p4);
      \draw (p3) -- (p4);
     \draw (p2) -- (p3);
      \draw (p2) -- (p1);
            \draw [ultra thick, ->,>=stealth, black!40!white](p1) -- (0.52,0.74);
      \draw [ultra thick, ->,>=stealth, black!40!white](p3) -- (1.88,-0.84);
      \draw [ultra thick, ->,>=stealth, black!40!white](p2) -- (-0.5,-0.65);
      \draw [ultra thick, ->,>=stealth, black!40!white](p4) -- (2.9,0.55);
      \filldraw[fill=black, draw=black]
    (1.2,-0.05) circle (0.004cm);

    
    
 \node [draw=white, fill=white] (b) at (1.2,0) { $\circlearrowleft$};
      \node [draw=white, fill=white] (a) at (1.5,-2.1) {(c)};
    \end{tikzpicture}
    \end{center}
\vspace{-0.3cm}
\caption{Symmetric frameworks with $C_4$ as underlying graph. (a) is $\mathcal{C}_{4v}$-symmetric (and hence $\tau(\Gamma)$-symmetric for any subgroup $\tau(\Gamma)$ of $\mathcal{C}_{4v}$) and  (b) and (c) are $\mathcal{C}_s$-symmetric (with respect to the reflection $\sigma$) and $\mathcal{C}_2$-symmetric, respectively. The framework in (c) has a non-trivial $\mathcal{C}_2$-symmetric infinitesimal motion, which extends to a continuous symmetry-preserving motion. }
\label{symfwks}
\end{figure}

A key motivation for studying $\tau(\Gamma)$-symmetric infinitesimal rigidity is that a $\tau(\Gamma)$-symmetric infinitesimal motion extends to a continuous, symmetry-preserving motion of the framework, provided that the configuration is sufficiently generic with the given symmetry constraints \cite{sch10}. In the present paper, we will exploit the notion of $\tau(\Gamma)$-symmetric infinitesimal rigidity to reduce the communication/sensing requirements in multi-agent formations.

One of the block matrices of the block-decomposed rigidity matrix (the block matrix $R_0(\G,p)$ corresponding  to the trivial irreducible representation of the group $\tau(\Gamma)$ which assigns $1$ to each group element) has the property that its kernel is isomorphic to the space of $\tau(\Gamma)$-symmetric infinitesimal motions and its cokernel is isomorphic to the space of $\tau(\Gamma)$-symmetric self-stresses. Thus, in the study of forced-symmetric (infinitesimal) rigidity, this block matrix plays the same role as the rigidity matrix in the standard non-symmetric theory.
In the next section we will introduce the \emph{orbit rigidity matrix}", which is equivalent to the block matrix $R_0(\G,p)$ and whose entries have a simple form similar to the entries of the standard rigidity matrix.

\subsection{The orbit rigidity matrix}\label{sec.orbitrigidity}

Given a $\tau(\Gamma)$-symmetric framework $(\G,p)$, it requires a non-trivial computation to obtain the block matrices of the block-decomposed rigidity matrix of $(\G,p)$. However, it was shown in \cite{sw2011} that the block matrix $R_0(\G,p)$  describing the $\tau(\Gamma)$-symmetric infinitesimal rigidity of $(\G,p)$ is equivalent to another matrix, called the \emph{orbit rigidity matrix}, which can be written down in a simple and direct way, without using any group representation theory.

For simplicity  we will assume from now on that the symmetry is always free (recall Definition~\ref{def:free}). While all of our results are expected to extend to the non-free case, this assumption  simplifies the definition of the orbit rigidity matrix, and hence the entire forced-symmetric formation control theory, significantly. So we will leave the non-free case for future work.
\begin{assumption}\label{ass.free}
The $\tau(\Gamma)$-symmetric framework $(\G,p)$ is free.
\end{assumption}

To describe the orbit rigidity matrix we first need to introduce the notion of a quotient gain graph of a $\Gamma$-symmetric graph, which in turn relies on the notions of vertex and edge orbits introduced in Definition~\ref{def:orb}.

\begin{definition}
Let $\G$ be $\Gamma$-symmetric graph with representative vertex set $\V_0$ and representative edge set $\E_0$. 
The \emph{quotient $\Gamma$-gain graph} $\G_0$ of $\G$ is the directed multigraph with vertex set $\V_0$ whose edge set $\E_0$ has the directed edge $(i,j)$ with group label (or \emph{gain}) $\gamma\in\Gamma$, denoted by $((i,j);\gamma)$, for each edge orbit representative $i\gamma(j)$.
\end{definition}




Note that different choices of vertex representatives give rise to different, but equivalent quotient $\Gamma$-gain graphs. Moreover, for a fixed choice of $\V_0$, the edge $((i,j);\gamma)$ in $\E_0$, where $i\neq j$, is equivalent to the edge $((j,i);\gamma^{-1})$, so the orientation of non-loop edges in $\E_0$ may be reversed by changing the gain to its inverse.  

\begin{example}
The framework $(C_4,p)$ in Figure~\ref{symfwks}(a) has $\mathcal{C}_{4v}$ symmetry, but we may consider it as a framework with the smaller rotational $\mathcal{C}_4$ symmetry (where $\mathcal{C}_4$ is generated by the $90^\circ$ rotation  about the origin).
Let $\Gamma=\{\mathrm{id},\psi_1,\psi_1^2,\psi_1^3\}$ be the corresponding subgroup of $\aut(C_4)$, where $\psi_1$ is the automorphism defined in Example~\ref{ex:aut}. For simplicity,  we will identify the groups $\aut(C_4)$ and $\mathcal{C}_4$ in the following discussion. Note that the symmetry $\mathcal{C}_4$ is free and there is only one vertex orbit under $\mathcal{C}_4$.  We may pick vertex $1$ as its representative. Then there is only one edge orbit, represented by the edge $13$, for example, which in the quotient $\mathcal{C}_4$-gain graph of $C_4$ corresponds to the loop at $1$ with gain $\psi_1$. The quotient $\mathcal{C}_4$-gain graph of $C_4$ is shown in Figure~\ref{symfwksquot}(a).

For the $\mathcal{C}_s$-symmetric framework in Figure~\ref{symfwks}(b), the corresponding subgroup $\Gamma$ of $\aut(C_4)$ is the group consisting of the identity and $\psi_4=(12)(34)$, as defined in Example~\ref{ex:aut}. Again we identify $\Gamma$ and $\mathcal{C}_s$. Since $\psi_4$ has two cycles (each of length $2$), there are two vertex orbits. We may pick the representatives $1$ and $3$ for these vertex orbits. Then there are three edge orbits, one of size $2$ represented by the edge $13$, and two of size $1$, consisting of the edges $12$ and $34$, respectively. This leads to the  quotient $\mathcal{C}_s$-gain graph of $C_4$  shown in Figure~\ref{symfwksquot}(b). Note that the directed edge $(3,1)$ has the identity gain  since it joins two vertex orbit representatives, whereas the loops have the non-trivial gain $\psi_4$. 

Finally, the corresponding subgroup $\Gamma$ of $\aut(C_4)$
for the $\mathcal{C}_2$-symmetric framework in Figure~\ref{symfwks}(c) consists of the identity and the automorphism $\psi_2=(14)(23)$. In this case there are again two vertex orbits, which we may represent by the vertices $1$ and $2$, and two edge orbits, represented by $12$ and $13$.   The corresponding quotient $\mathcal{C}_2$-gain graph of $C_4$ is shown in Figure~\ref{symfwksquot}(c). It has two parallel edges from $1$ to $2$, one with identity gain and one with gain $\psi_2$.
\end{example}

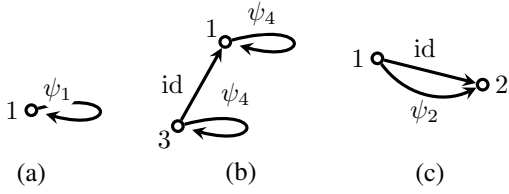
\begin{figure}[htp]
\begin{center}
\begin{tikzpicture}[very thick,scale=0.5]
\tikzstyle{every node}=[circle, draw=black, fill=white, inner sep=0pt, minimum width=4pt];

\path (0,0) node (p1) [label = left: $1$] {} ;

    \path(p1) edge [loop right,->, >=stealth,shorten >=3pt,looseness=60] (p1);  
        
        \node [draw=white, fill=white] (a) at (0.7,0.55) {$\psi_1$};
        
         \node [draw=white, fill=white] (a) at (0,-1.7) {(a)};

                  \end{tikzpicture}
                  \hspace{0.2cm}
          \begin{tikzpicture}[very thick,scale=0.7]
           \tikzstyle{every node}=[circle, draw=black, fill=white, inner sep=0pt, minimum width=4pt];
            \path (-0.7,0.8) node (p1) [label = above left: $1$] {};
            \path (-1.6,-0.8) node (p3) [label = below left: $3$] {} ;
             \draw[>=stealth,<-](p1)--(p3);

              \path(p1) edge [loop right,->, >=stealth,shorten >=3pt,looseness=60] (p1);  
         \path(p3) edge [loop right,->, >=stealth,shorten >=3pt,looseness=60] (p3);
         
       \node [draw=white, fill=white] (a) at (0,1.35) {$\psi_4$}; 
       \node [draw=white, fill=white] (a) at (-0.5,-0.2) {$\psi_4$};
       \node [draw=white, fill=white] (a) at (-1.7,0) {$\mathrm{id}$};
        
         \node [draw=white, fill=white] (a) at (-0.4,-1.7) {(b)};
        
                  \end{tikzpicture}  
                  \hspace{0.2cm}
        \begin{tikzpicture}[very thick,scale=0.7]
\tikzstyle{every node}=[circle, draw=black, fill=white, inner sep=0pt, minimum width=4pt];
\node [draw=white, fill=white] (a) at (1.4,0.7) {$\mathrm{id}$};
    \node [draw=white, fill=white] (a) at (1.4,-0.5) {$\psi_2$};
 \path (0.5,0.5) node (p1) [label = left: $1$] {} ;
       \path (2.5,0) node (p2) [label = right: $2$] {} ;
 \draw[>=stealth,->] (p1) -- (p2);
    \draw(p1) edge [->,>=stealth,bend right=40] (p2);

  
    
 
      \node [draw=white, fill=white] (c) at (1.5,-1.7) {(c)};
    \end{tikzpicture}          
   \end{center}
\vspace{-0.3cm}
\caption{The quotient $\Gamma$-gain graphs of the  graphs of the frameworks in Fig.~\ref{symfwks}. }
\label{symfwksquot}
\end{figure}

We are now ready to define the orbit rigidity matrix which describes the $\tau(\Gamma)$-symmetric infinitesimal rigidity properties of a $\tau(\Gamma)$-symmetric framework. In particular, as we will see in Theorem~\ref{thm:omatrix}, its kernel and cokernel are isomorphic to the space of $\tau(\Gamma)$-symmetric infinitesimal motions and self-stresses of $(\G,p)$, respectively.

\begin{definition} \label{def:om}
Let $\G$ be a $\Gamma$-symmetric graph, where the symmetry is free, and let $(\G,p)$ be a $\tau(\Gamma)$-symmetric framework. Further, let $\G_0=(\V_0,\E_0)$ be the quotient $\Gamma$-gain graph of $\G$ and denote $\bar{p}=p|_{\V_0}$, the restriction of the configuration to only the representative vertices in $\V_0$.  Then the \emph{orbit rigidity matrix} $\mathcal{O}(\G_0,\bar{p})$ of $(\G,p)$ is the $|\E_0|\times d|\V_0|$ matrix defined as follows. The row corresponding to an edge $((i,j);\gamma)$, where $i\neq j$, has the  form:
 {\footnotesize $$\left(\begin{array}{ccccc}   0  \cdots  0 & (\bar p_{i}-\tau(\gamma)\bar p_{j})^T & 0  \cdots  0 & (\bar p_{j}-\tau(\gamma)^{-1}\bar p_{i})^T &  0   \cdots  0 \end{array} \right),
$$}with the $d$-dimensional entries  $(\bar p_{i}-\tau(\gamma)\bar p_{j})^T$ and $(\bar p_{j}-\tau(\gamma)^{-1}\bar p_{i})^T$ being in the columns corresponding to vertex $i$ and $j$, respectively.
The row corresponding to a loop $((i,i);\gamma)$ has the  form:
 $$\left(\begin{array}{ccc}   0  \cdots  0 & (2\bar p_{i}-\tau(\gamma)\bar p_{i}-\tau(\gamma)^{-1}\bar p_i)^T & 0  \cdots  0   \end{array} \right),
$$
with the $d$-dimensional entry  $(2\bar p_{i}-\tau(\gamma)\bar p_{i}-\tau(\gamma)^{-1}\bar p_i)^T$ being in the columns corresponding to  vertex $i$.
\end{definition}

\begin{example}\label{ex:oma}
The framework in Figure~\ref{symfwks}(a) has the quotient $\mathcal{C}_4$-gain graph shown in Figure~\ref{symfwksquot}(a). The matrix representing the rotation $c_4$ by $90^\circ$ is given by $\begin{bmatrix} 0& -1\\ 1&0\end{bmatrix}$ and its inverse is $\begin{bmatrix}  0& 1\\ -1&0\end{bmatrix} $. So for $p_1=\begin{bmatrix} x_1&y_1\end{bmatrix} ^T$, the orbit rigidity matrix is the $1\times 2$ matrix 
$$\mathcal{O}(\G,\bar p)=\begin{bmatrix}  2x_1 & 2y_1 \end{bmatrix} .
$$
Note that the kernel of this matrix is the  $1$-dimensional space spanned by the rotational vector $\begin{bmatrix}-y_1 & x_1\end{bmatrix}^T$ at $p_1$, which corresponds to the $\mathcal{C}_4$-symmetric infinitesimal rotation of the framework via \eqref{eq:symmot}. So this confirms that the framework is $\mathcal{C}_4$-symmetric infinitesimally rigid. Since the cokernel is trivial, the framework has no non-trivial $\mathcal{C}_4$-symmetric self-stress, and hence the framework is $\mathcal{C}_4$-symmetric isostatic. 

The framework in Figure~\ref{symfwks}(b) has the quotient $\mathcal{C}_s$-gain graph shown in Figure~\ref{symfwksquot}(b). The matrix representing the reflection $\sigma$ is given by $\begin{bmatrix}  -1& 0\\ 0&1\end{bmatrix} $ and this matrix is equal to its inverse. So for  $p_i=\begin{bmatrix} x_i &y_i \end{bmatrix}^T $, $i=1,3$, the orbit rigidity matrix is the $3\times 4$ matrix 
$$\mathcal{O}(\G,\bar p)=\begin{bmatrix}  x_1-x_3 & y_1-y_3 & x_3-x_1 & y_3-y_1\\
4x_1 & 0 & 0 & 0\\
0& 0& 4x_3 &0\end{bmatrix}.
$$
The kernel of this matrix is the  $1$-dimensional space spanned by the vector $\begin{bmatrix}0 & 1 & 0 & 1\end{bmatrix}^T$, which corresponds to the $\mathcal{C}_s$-symmetric vertical infinitesimal translation of the framework via \eqref{eq:symmot}. Thus, the framework is $\mathcal{C}_s$-symmetric infinitesimally rigid. In fact, it is $\mathcal{C}_s$-symmetric isostatic, since the cokernel is again trivial. 

Note that since both frameworks are $\tau(\Gamma)$-symmetric isostatic, the removal of any edge orbit, or equivalently the removal of any edge in the quotient $\Gamma$-gain graph, yields a $\tau(\Gamma)$-symmetric infinitesimally flexible framework.

Finally, for the framework in Figure~\ref{symfwks}(c), the quotient $\mathcal{C}_s$-gain graph is shown in Figure~\ref{symfwksquot}(c). The matrix representing the half-turn  is given by $\begin{bmatrix}  -1& 0\\ 0&-1\end{bmatrix} $ and this matrix is equal to its inverse. For $p_i=\begin{bmatrix}x_i &y_i\end{bmatrix}^T$, $i=1,2$, the orbit rigidity matrix is the $2\times 4$ matrix  $$\mathcal{O}(\G,\bar p)=\begin{bmatrix}  x_1-x_2 & y_1-y_2 & x_2-x_1 & y_2-y_1\\
x_1+x_2 & y_1+y_2 & x_2+x_1 & y_2+y_1\\
\end{bmatrix}.$$
The $2$-dimensional kernel of this matrix consists of a $1$-dimensional space of vectors that correspond to infinitesimal rotations, and a $1$-dimensional space of vectors that correspond to non-trivial $\mathcal{C}_2$-symmetric infinitesimal motions of the framework (see Figure~\ref{symfwks}(c)).
\end{example}

The key properties of the orbit rigidity matrix,  established in \cite{sw2011}, are summarized in the following theorem.

\begin{theorem} \label{thm:omatrix}
    Let $(\G,p)$ be a $\tau(\Gamma)$-symmetric framework with orbit rigidity matrix $\mathcal{O}(\G_0,\bar{p})$. Then, 
    \begin{itemize} 
    \item[(i)] the kernel of $\mathcal{O}(\G_0,\bar{p})$ is isomorphic to  the space of $\tau(\Gamma)$-symmetric infinitesimal motions of $(\G,p)$, and
    \item[(ii)] the cokernel of $\mathcal{O}(\G_0,\bar{p})$ is isomorphic to  the space of $\tau(\Gamma)$-symmetric self-stresses of $(\G,p)$.
    \end{itemize}
\end{theorem}

More precisely, an element ${\bar u}$  in the kernel of $\mathcal{O}(\G_0,\bar{p})$ assigns a velocity vector to each vertex orbit representative in $\V_0$. From this we can uniquely construct the corresponding $\tau(\Gamma)$-symmetric infinitesimal motion of $(\G,p)$ via \eqref{eq:symmot}. Conversely, if we restrict any $\tau(\Gamma)$-symmetric infinitesimal motion of $(\G,p)$ to $\V_0$, then we obtain a vector in the kernel of $\mathcal{O}(\G_0,\bar{p})$.
Similarly, each element ${\bar \omega}$  in the cokernel of $\mathcal{O}(\G_0,\bar{p})$ assigns a scalar to each edge orbit representative in $\E_0$. From this we can uniquely construct the corresponding $\tau(\Gamma)$-symmetric self-stress of $(\G,p)$ via \eqref{eq:symss}, and vice versa. See \cite{sw2011} for details. 

It is a simple consequence of Theorem~\ref{thm:omatrix} that if a $\tau(\Gamma)$-symmetric framework $(\G,p)$ is $\tau(\Gamma)$-symmetric infinitesimally rigid, then the same is true for all $\tau(\Gamma)$-symmetric frameworks $(\G,q)$ in an open neighborhood of  $p$ (and in fact for an open dense subset of the set of $\tau(\Gamma)$-symmetric realisations of $\G$ as a bar-joint framework).  

Using the orbit rigidity matrix, efficient Laman-type combinatorial characterisations of the graphs that `generically' yield $\tau(\Gamma)$-symmetric infinitesimally rigid frameworks in the plane have been established for $\mathcal{C}_s$, $\mathcal{C}_n$ and $\mathcal{C}_{(2n+1)v}$, $n\geq 1$ (in the case where the symmetry is free) in \cite{jtk}. The problem remains open for the dihedral groups $\mathcal{C}_{2nv}$. Moreover, as in the non-symmetric situation, there are no analogues of these results in higher dimensions.

Of importance to this work is a natural corollary of Theorem \ref{thm:omatrix} describing the rank of the orbit rigidity matrix.

\begin{corollary}\label{cor.orbitisostatic}
    Let $(\G,p)$ be a $\tau(\Gamma)$-symmetric independent framework with orbit rigidity matrix $\mathcal{O}(\G_0,\bar{p})$. Then $\mathcal{O}(\G_0,\bar{p})$ has full row-rank.
\end{corollary}

In particular, by Corollary~\ref{cor.orbitisostatic},  the orbit rigidity matrix $\mathcal{O}(\G_0,\bar{p})$  of a $\tau(\Gamma)$-symmetric isostatic framework $(\G,p)$ has full row-rank. In addition, it has the desired rank to guarantee $\tau(\Gamma)$-symmetric infinitesimal rigidity. This rank depends on the point group of $(\G,p)$ \cite{jtk}.

It is useful to keep in mind the parallels between the standard rigidity matrix $R(\mathcal G,p)$ and the orbit rigidity matrix.  Corollary \ref{cor.orbitisostatic} can thus be thought of as the symmetric counterpart to the known result that the rigidity matrix of an independent  framework has full row-rank. If the framework is isostatic then it also has the desired rank to guarantee infinitesimal rigidity (namely $2n-3$ in the plane).

\section{Forced Symmetric Formation Control}\label{sec.symform}

We now would like to study a variation of the formation control problem where the goal of each agent in the network is to obtain a formation corresponding to a \emph{symmetric configuration}.  In other words, starting with a $\Gamma$-symmetric graph $\G$, which can be drawn with maximum point group symmetry $\mathcal{S}$ in $\mathbb{R}^d$, we would like to drive the agents to a special position $p^\star$ such that the framework $(\G,p^\star)$ is $\tau(\Gamma)$-symmetric for a desired subgroup of $\mathcal{S}$.

\begin{problem}\label{problem_symaquisition}

Consider a group of $n$ integrator agents \eqref{integrator} that interact over the $\Gamma$-symmetric sensing graph $\G$.  Let $p^\star \in \mathbb{R}^{dn}$ be a configuration such that $(\G,p^\star)$ is $\tau(\Gamma)$-symmetric for some desired point group $\tau(\Gamma)$, and let $\V_0$ be a set of representatives of the vertex orbits of $\G$ under $\Gamma$.  Design a control $u_i(t)$ for each agent $i$ 
such that 
\begin{itemize}
    \item[(i)] $\underset{t\to \infty }{\lim} \|p_i(t)-p_j(t)\| = \|p^\star_i-p^\star_j\|$ for all $ij\in \E$;
    \item[(ii)] for each $i\in \mathcal{V}_0$,  $\underset{t\to \infty }{\lim} \|p_u(t)-\tau(\gamma_{vu})p_v(t)\| = 0$ for all   $u,v \in \Gamma_i$. 
\end{itemize}
\end{problem}

We would like to solve Problem \ref{problem_symaquisition} in a distributed fashion, ideally allowing agent $i$ to only obtain information from neighboring agents as defined by $\G$.  Before we proceed, we comment on the information needed to solve this problem.


Requirement (i) in Problem \ref{problem_symaquisition} is the standard formation control constraint introduced in \S\ref{sec.formationcontrol}.  That is, $\|p^\star_i(t)-p^\star_j(t)\|=\mathrm{\bf d}_{ij}$ are the desired distances between neighboring agents. Requirement (ii) aims to enforce the symmetric position between agents that are in the same vertex orbit. The statement of Problem \ref{problem_symaquisition} implicitly assumes that the requirements (i) and (ii) are consistent with each other.

We also point out that it may not be the case that  $u,v \in \Gamma_i$ implies that $uv \in \E$.  
{\color{black}{We will show that the following assumption on the subgraph induced by the vertex orbits is sufficient to realize this constraint.
\begin{assumption}\label{vertexorbit_subgraph}
    The sub-graph induced by each vertex orbit $\Gamma_i$, denoted $\mathcal{G}(\Gamma_i)$, is connected.
\end{assumption}
\vspace{-.5cm}
\textcolor{black}{
\begin{remark}
    Assumption~\ref{vertexorbit_subgraph} is required to ensure all agents within the same vertex orbit eventually exchange information with each other.  This is analogous to standard connectivity assumptions in consensus algorithms, of which this is a generalization \cite{Mesbahi2010}. Relaxation of this assumption would necessitate a more complicated control architecture including distributed observers to estimate the missing relative state information.
\end{remark}
}

The vertex orbits also introduce a natural labeling for the agents in the network.  Without loss of generality, the vertex orbit $\Gamma_i$ will contain the vertices $\{i,i+1,\ldots,i+|\Gamma_i|-1\}$.  Furthermore, we denote by $\mathcal{E}(\Gamma_i)\subset \mathcal E$  the edges in $\mathcal{G}(\mathcal E_i)$.  Similarly, the state-vector $p(t)$ can also be partitioned as $p(t) = \begin{bmatrix} p^1(t)^T & \cdots & p^{|\V_0|}(t)^T\end{bmatrix}^T$ where $p^i(t) \in \mathbb{R}^{d|\Gamma_i|}$, and $p^{i}_1(t)$ is the position associated to vertex $i\in \V_0$.

}}




\subsection{Symmetry potentials and formation stabilization}\label{sec.formationstab}
Our approach to solve Problem \ref{problem_symaquisition} follows the same gradient dynamical system approach used in solving the standard formation control problem, as introduced in Section \ref{sec.formationcontrol}.  We provide first a brief summary of our result from \cite{Zelazo_IFAC2023}, together with some new results.  In this direction, we define a \emph{symmetry-forcing potential},
\textcolor{black}{\begin{align}\label{sym_potential}
\hspace{-12pt}F_s(p(t)) =  \frac{1}{2}\sum_{i \in \V_0} \sum_{\substack{u,v\in \Gamma_i\\uv\in \mathcal{E}}} \|p_u(t) - \tau(\gamma_{vu} )p_v(t)\|^2.
\end{align}}
Here, recall that $\V_0$ is a set of representatives from each of the vertex orbits of a $\Gamma$-symmetric graph.
The \emph{symmetric formation potential} can then be defined as 
\begin{align}\label{sym_form_potential}
\hspace{-12pt}F(p(t)) = F_f(p(t)) + F_s(p(t)),
\end{align}
where $F_f(p(t))$ is the formation potential defined in \eqref{formation_potential}.

We now propose the control
\begin{align}\label{symform_acquire}u(t) = -\nabla F(p(t)),
\end{align}
to solve Problem \ref{problem_symaquisition}.  The closed-loop dynamics then take the form
\begin{align}\label{ctrl_1}\dot{p}(t) = -R(\G,p(t))^T\left(R(\G,p(t))p(t)-\mathrm{\bf d}^2\right) - Qp(t),
\end{align}
where $Q$ is a block diagonal matrix with $|\V_0|$ blocks, where each block is $|\Gamma_i|d \times |\Gamma_i|d$.  With the labeling  of the nodes defined earlier, $Q$ can always be written as
{\small $$Q = \begin{bmatrix} {Q}_1 && \\ &\ddots & \\ && {Q}_{|\V_0|}\end{bmatrix},$$}
and
\textcolor{black}{$$ [Q_i]_{uv} = \begin{cases}
                d_{\Gamma_i}(u)I, & u=v, \, u \in \Gamma_i \\
                -\tau(\gamma_{uv}), & uv\in \mathcal{E}, u,v\in \Gamma_i \\
                0, & \text{o.w.}
            \end{cases},$$
where $d_{\Gamma_i}(u)$ denotes the degree of node $u$ in the induced subgraph $\G(\Gamma_i)$. Observe that $Q_i$ can be expressed as the matrix product $E(\Gamma_i)E(\Gamma_i)^T$, where $E(\Gamma_i)$ has a similar structure as the incidence matrix.  For an edge $k=uv$ with $u,v\in\Gamma_i$, one has  $[E(\Gamma_i)]_{uk}=I$ if the edge $k$ leaves vertex $u$, $[E(\Gamma_i)]_{uk}=-\tau(\gamma_{vu})$ if the edge $k$ enters the vertex $u$, and $[E(\Gamma_i)]_{uk}=0$ otherwise.  Therefore $Q_i$ (and consequently $Q$) is a positive semi-definite matrix.  To further simplify notation, we define $\bar E(\Gamma) = \mathrm{diag}\{E(\Gamma_i)\}_{i=1:|\V_0|}$, and therefore $Q=\bar E(\Gamma)\bar E(\Gamma)^T$. It follows that any configuration that is in a symmetric position must lie in the kernel of $Q$.
}

%
    

\textcolor{black}{It is also useful to examine the expression of the closed-loop dynamics for each agent.  For an agent $i$ in the vertex orbit $\Gamma_u$, the dynamics are
\begin{align*}
    \dot{p}_i(t) &= \sum_{ij\in \E}(\|p_i(t)-p_j(t)\|^2-(\mathrm{\bf d}_{ij})^2)(p_j(t)-p_i(t)) \\
    &+ \sum_{\substack{ij\in\E\\i,j\in \Gamma_u}}(\tau({\gamma_{ij}})p_j(t)-p_i(t)).
\end{align*}
}
{\color{black}{We now show that the closed-loop dynamics \eqref{ctrl_1} has an invariant quantity.
\begin{proposition}
    Consider the closed-loop dynamics \eqref{ctrl_1} and let 
    \begin{align}\label{group_avg}
        z(t) &= \sum_{v \in \V} \sum_{\gamma \in \Gamma} \tau(\gamma)p_v(t).
    \end{align}
    Then $z(t)$ remains invariant along the trajectories of \eqref{ctrl_1}, i.e., $\dot{z}(t) = 0$.
\end{proposition}
\begin{proof}
    We examine the derivative of $z(t)$,
    \begin{align}\label{zdot}
        \dot{z}(t) &=\sum_{v \in \V} \sum_{\gamma \in \Gamma} \tau(\gamma) \dot{p}_v(t).
    \end{align}
%
    We will study separately the contribution of the distance constraint term and the symmetry-forcing term from each agent in \eqref{zdot}. 

    To begin, we examine the symmetry-forcing contribution to $\dot z(t)$ from an edge $ij\in\E$ with $i,j\in \Gamma_u$ for any $u\in \V_0$. This leads to 
    \begin{align*}
    &\sum_{\gamma\in \Gamma} \tau(\gamma) \left(
    \tau(\gamma_{ij})p_j-p_i
    \right) + \sum_{\gamma\in \Gamma} \tau(\gamma) \left(
    \tau(\gamma_{ji})p_i-p_j
    \right)   &   \\& = 
    \left(\left(\sum_{\gamma\in \Gamma} \tau(\gamma) \right) \tau(\gamma_{ij}) -\sum_{\gamma\in \Gamma} \tau(\gamma) 
\right)p_j(t) \\
&+\left(\left(\sum_{\gamma\in \Gamma} \tau(\gamma) \right) \tau(\gamma_{ji}) -\sum_{\gamma\in \Gamma} \tau(\gamma) 
\right)p_i(t)=0,
    \end{align*}
    where the last equation follows from the fact that 
$\left(\sum_{\gamma\in \Gamma} \tau(\gamma) \right) \tau(\gamma') =\sum_{\gamma\in \Gamma} \tau(\gamma)$
for any $\gamma'\in \Gamma$.  Note that this holds for each  edge connecting agents in the same vertex orbit.

We now look at the contribution from the distance constraint terms in the agent dynamics.  Consider  again  the edge $ij\in \E$ (with no restriction on agent $j$ being in the same vertex orbit of agent $i$).  Let $a_{ij} = a_{ji} =\|p_i-p_j\|^2-\mathrm{\bf d}_{ij}^2$.  Then it is straightforward to verify that
\begin{align*}
    \sum_{\gamma\in \Gamma}\tau(\gamma)a_{ij}(p_j-p_i) + \sum_{\gamma\in \Gamma}\tau(\gamma)a_{ji}(p_i-p_j)=0.
\end{align*}
This also holds for each edge in $\E$.  Together this shows that $\dot{z}(t) = 0$ as claimed.
\end{proof}
}}

We now present the main result of \cite{Zelazo_IFAC2023}.
\begin{theorem}[\cite{Zelazo_IFAC2023}]\label{IFAC_Thm}
Consider a team of $n$ agents \eqref{integrator} interacting over a $\Gamma$-symmetric graph $\G$ satisfying Assumption \ref{vertexorbit_subgraph}, that can be drawn with maximum point group symmetry $\mathcal{S}$ in $\mathbb{R}^d$, and let 
$$\mathcal{F}_f = \{p \in \mathbb{R}^{dn} \, | \, \|p_i-p_j\|=\mathrm{\bf d}_{ij} \, ij\in \E\},$$  and
$$\mathcal{F}_s = \{p \in \mathbb{R}^{dn} \, | \, \tau(\gamma) (p_i)=p_{\gamma (i)} \, \forall \gamma\in \Gamma, \,  i\in \V\}.$$
Then for initial conditions $p_i(0)$ satisfying 
{\small $$\sum_{ij\in\E} (\|p_i(0)-p_j(0)\| - \mathrm{\bf d}_{ij}^2) \leq \epsilon_1,\text{and } \|p_i(0)-\tau(\gamma_j)p_j(0)\|^2\leq \epsilon_2$$}
for all $i\in \V_0$ and $j\in \Gamma_i$,
for a sufficiently small and positive constant $\epsilon_1$ and $\epsilon_2$, the control 
\begin{align}\label{formation_u}
u = -\nabla F(p(t)),
\end{align}
renders the set $\mathcal{F}_f \cap \mathcal{F}_s$ exponentially stable, i.e.
$$\underset{t\to \infty}{\lim} \|p_i(t)-p_j(t)\| = \mathrm{\bf d}_{ij}
\text{ and } \underset{t\to \infty}{\lim}\tau(\gamma) (p_i(t))=\underset{t\to \infty}{\lim} p_{\gamma (i)}(t)$$ for all  $\gamma\in \Gamma, i\in\V.$
\end{theorem}
The implementation of the control \textcolor{black}{\eqref{formation_u}} assumes that agents within the same vertex orbit are able to exchange information with each other (i.e., Assumption \ref{vertexorbit_subgraph} is satisfied). This may not be the case for different symmetries, as illustrated in the next example.


\begin{example}\label{ex.traj2}
Consider the graph in Figure \ref{fig:graphex2}.  This is a $\tau(\Gamma)$-symmetric framework with the point group symmetry specified by the reflection in the $y$-axis.  The vertex orbits  for this symmetry are $\Gamma_1=\Gamma_2=\{1,2\}$, $\Gamma_3=\Gamma_6=\{3,6\}$, and $\Gamma_4=\Gamma_5=\{4,5\}$.  Note that the nodes in the orbit $\Gamma_3$ are not connected by edges in $\G$.  Since Assumption \ref{vertexorbit_subgraph} is not satisfied, implementation of  \eqref{symform_acquire}  requires establishing additional communication between nodes $3$ and $6$.

\begin{figure}[!h]
\begin{center}
\begin{tikzpicture}[very thick,scale=0.7]
\tikzstyle{every node}=[circle, draw=black, fill=white, inner sep=0pt, minimum width=5pt];
\path (1.1,1.1) node (p1) [label = above left: ${1}$] {} ;
\path (1.9,1.1) node (p11) [label = above right: ${2}$] {} ;
       
        \path (2.7,0.5) node (p2) [label = above right: ${3}$] {} ;
        \path (2.4,-.75) node (p3) [label = below right: ${4}$] {} ;
        \path (0.6,-.75) node (p4) [label = below left: ${5}$] {} ;
        \path (0.3,.5) node (p5) [label = above left: ${6}$] {} ;

    \draw (p1) -- (p11);
    \draw (p11) -- (p2);
    \draw (p2) -- (p3);
    \draw (p3) -- (p4);
    \draw (p4) -- (p5);
    \draw (p1) -- (p5);
    \draw (p1) -- (p4);
    \draw (p11) -- (p3);
       
        \draw [dashed, thin] (1.5,-1.6) -- (1.5,1.6);
        \node [draw=white, fill=white] (b) at (1.8,-1.6) {$\sigma$};
    \end{tikzpicture}\vspace{-.35cm}
    \caption{A $\tau(\Gamma)$-symmetric graph with $y$-axis symmetry not satisfying Assumption \ref{vertexorbit_subgraph}.}\label{fig:graphex2}
\end{center}\vspace{-.7cm}
\end{figure}
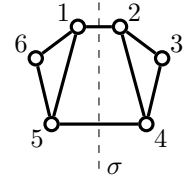


\end{example}

{\color{black}{
We also note that the implementation of \eqref{ctrl_1} requires $|\E|$ communication/sensing links for each distance constraint, and  that the sub-graph induced by the vertex orbits are connected.  In fact, Assumption \ref{vertexorbit_subgraph} can be relaxed further to require only a spanning tree for each sub-graph induced by the vertex orbits.  
\begin{corollary}
    The conditions stated in Theorem \ref{IFAC_Thm} hold if the induced sub-graph of each vertex orbit, $\G(\Gamma_i)$, is a spanning tree.
\end{corollary}
The proof remains unchanged from Theorem \ref{IFAC_Thm}.  The corollary shows that symmetry between agents in the same vertex orbit, say $\Gamma_i$, can be maintained with only $|\Gamma_i|-1$ edges.  Nevertheless, the control proposed in Theorem \ref{IFAC_Thm} still requires to use all the edges encoding the distance constraints.  We now show that this too is redundant and only edges from the representative edge set $\E_0$ need to be considered.

\subsection{Orbit rigidity potentials and formation stabilization}
In this direction, we define a new potential function that aims to satisfy the distance constraints over the representative edges in the edge orbits,
$$F_e(p(t)) = \frac{1}{4} \sum_{e=ij \in \E_0} \left( \| p_i - \tau(\gamma)p_j \|^2-\mathrm{\bf d}_{i\gamma(j)}^2\right)^2,$$
where $\gamma \in \Gamma$ is the label of the edge in the quotient gain graph.
Thus, the new forced-symmetric formation control potential can be expressed as 
\begin{align}\label{sympotential2}F(p(t)) = F_e(p(t)) + F_s(p(t)),
\end{align}
where $F_s(p(t))$ was defined in \eqref{sym_potential}.
As before, we now propose the control
$$u(t)  = -\nabla F(p(t)).$$
To help simplify notations, we denote by $p_0(t)$ the restriction of the configuration vector $p(t)$ to only  the agents in the representative vertex set $\V_0$.  The remaining agents are denoted by the vector $p_f(t)$, so with an appropriate labeling of the agents we can write $\tilde{p}(t) = Pp(t) = \begin{bmatrix} p_0^T(t) & p_f^T(t)\end{bmatrix}^T$, for some permutation matrix $P$. 
Then the control for each agent $i\in \V_0$ can be expressed as
\begin{align}\label{repnode_dynamics}
    u_i(t) &=  u_i^{(a)}(t)+u_i^{(b)}(t)+u_i^{(c)}(t),
\end{align}
where
{\footnotesize\begin{align*}
    u_i^{(a)}(t) &=\hspace{-5pt}\underset{\substack{i\gamma(j) \in \E_0\\ j\in \V_0, \,i\neq j}}{\sum}\hspace{-8pt} \big(\|p_i(t)-\tau(\gamma)p_j(t)\|^2-\mathrm{\bf d}_{ij}^2\big) (\tau(\gamma)p_j(t)-p_i(t))\\
 u_i^{(b)}(t) &=\hspace{-5pt}\underset{\substack{i\gamma(i) \in \E_0}}{\sum}\hspace{-8pt}(\|(I-\tau(\gamma))p_i\|^2-\mathrm{\bf d}_{i\gamma(i)}^2)(2I-\tau(\gamma)-\tau(\gamma)^{-1})p_i \\
 u_i^{(c)}(t) &= \sum_{{ij\in\E(\Gamma_i)}}(\tau({\gamma_{ij}})p_j(t)-p_i(t)).
\end{align*}}
The control for the agents in $\V\setminus \V_0$ is simply
\begin{align}\label{folnode_dynamics}
    u_i(t)&= \sum_{{ij\in\E(\Gamma_v)}}(\tau({\gamma_{ij}})p_j(t)-p_i(t)), 
\end{align}
for each $v \in \V_0$.

When expressing the dynamics in a state-space form, the orbit rigidity matrix explicitly appears, and  we obtain
\begin{align}\label{orbit_symformation}
    \begin{bmatrix} \dot{p}_0(t) \\ \dot{p}_f(t) \end{bmatrix} &= \begin{bmatrix} -\mathcal{O}^T(\mathcal{G}_0,p_0(t))\bigg(\mathcal{O}(\mathcal{G}_0,p_0(t))p_0(t) - \mathrm{\bf d}_0^2 \bigg) \\0\end{bmatrix} \nonumber\\&- PQP^T\begin{bmatrix}p_0(t) \\ p_f(t)\end{bmatrix}.
\end{align}
Here $\mathrm{\bf d}_0$ are the distance constraints for the edges in $\E_0$.  
It is interesting to observe the structure of this system.  The orbit rigidity matrix in \eqref{orbit_symformation} plays the role of the rigidity matrix for \eqref{formation_control} in preserving the distances.  

{\color{black}{
It is now useful to define an error system to study the stability and convergence properties of the system.  In this direction, we let 
$$\bar{\sigma}(t) = \mathcal{O}(\mathcal{G}_0,p_0(t))p_0(t) - \mathrm{\bf d}_0^2\text{, and }\bar q(t)=\bar E(\Gamma)^TP^Tp(t),$$ with $\bar{e}(t) = \begin{bmatrix} \bar \sigma(t)^T & \bar q(t)^T \end{bmatrix}^T$. Then the error dynamics can be expressed as
\begin{align}\label{error_dynamics_orbit}
\begin{bmatrix} \dot{\bar \sigma }(t) \\ \dot{\bar q}(t) \end{bmatrix} &= -\begin{bmatrix} \mathcal{O}\mathcal{O}^T & \mathcal{O} \bar E_0(\Gamma) \\ \bar E_0^T(\Gamma)\mathcal{O}^T & \bar E^T(\Gamma)\bar E(\Gamma) \end{bmatrix} \begin{bmatrix} \bar \sigma(t) \\ \bar q(t) \end{bmatrix} \nonumber \\ 
&= -\begin{bmatrix} \begin{bmatrix} \mathcal{O} & 0 \end{bmatrix} \\ \bar E^T(\Gamma)P^T \end{bmatrix} \begin{bmatrix}\begin{bmatrix} \mathcal{O}^T \\ 0^T \end{bmatrix} P\bar E(\Gamma) \end{bmatrix} \begin{bmatrix} \bar \sigma(t) \\ \bar q(t) \end{bmatrix}.
\end{align}
%
Here, $\begin{bmatrix} \bar E_0(\Gamma)^T & \bar E_f(\Gamma)^T\end{bmatrix}^T = P\bar E(\Gamma)$.  We refer to \eqref{error_dynamics_orbit} as the \emph{orbit error system}. For notational simplicity we have dropped the explicit dependence of the orbit rigidity matrix on $\mathcal G_0$ and $p_0(t)$. We also can characterize its equilibria by the set
\begin{align}\label{orbit_Error_eq} \mathcal X_0 = \left\{(\bar \sigma, \bar q) \, | \, \mathcal{O}^T\bar \sigma + \bar E_0(\Gamma) \bar q = 0 \text{ and } \bar E_f(\Gamma) \bar q =0\right\}.
\end{align}
}}
Our first result shows that the set $\mathcal X_0$ is asymptotically stable.
%
\begin{theorem}\label{thm_orbiterrordynamics}
    Consider the error system \eqref{error_dynamics_orbit} for a $\tau(\Gamma)$-symmetric framework satisfying Assumption \ref{ass.free}.  The set of equilibria \eqref{orbit_Error_eq} is asymptotically stable. Furthermore, on the set $\mathcal X_0$, the control $u(t) = 0$ and the configuration vector $p(t)$ converges to a fixed point under the dynamics \eqref{orbit_symformation}.
\end{theorem}
\begin{proof}
     We define the Lyapunov function
    $$V(\bar e(t))= \frac{1}{2}\begin{bmatrix} \bar \sigma(t)^T & \bar q(t)^T\end{bmatrix} \begin{bmatrix} \bar \sigma(t) \\ \bar q(t) \end{bmatrix}=\frac{1}{2}\bar e(t)^T \bar e(t).$$
    The derivative of $V$ along the trajectories of \eqref{error_dynamics_orbit} is
    \begin{align*}
        \dot{V}(\bar e(t)) &= -\bar e(t)^T\underbrace{\begin{bmatrix} \begin{bmatrix} \mathcal{O} & 0 \end{bmatrix} \\ \bar E^T(\Gamma) P^T\end{bmatrix} \begin{bmatrix}\begin{bmatrix} \mathcal{O}^T \\ 0^T \end{bmatrix} P\bar E(\Gamma) \end{bmatrix}}_{\mathcal M}\bar e(t)\\
        &= -\left\|\begin{bmatrix}\begin{bmatrix} \mathcal{O}^T \\ 0^T \end{bmatrix} P\bar E(\Gamma) \end{bmatrix}\bar e(t)\right\|^2  \leq 0.
    \end{align*}  
    Note that for any $\rho>0$ the set $\Psi(\rho) = \{\bar e\,:\,V(\bar e)\leq \rho\}$ such that $\bar e(0) \in \Psi(\rho)$ is compact and positively invariant with respect to \eqref{error_dynamics_orbit}.  Therefore, by LaSalle's Theorem, every solution starting in $\Psi(\rho)$ must approach the largest invariant set where $\dot V(\bar e) = 0$, which is precisely the set $\mathcal X_0$.

    Now, observe from \eqref{orbit_symformation} that the control can be expressed as 
    $$u(t) = -\begin{bmatrix} \mathcal O^T  \\ 0 \end{bmatrix} \bar \sigma(t)- P\bar E(\Gamma ) \bar q.$$
    Therefore, on the set $\mathcal X_0$ it follows that $u(t) \equiv 0$.  Since $u(t)\to 0$ it follows that $p(t)$ must converge to a constant value.
\end{proof}

Theorem \ref{thm_orbiterrordynamics} guarantees that both the orbit error dynamics and the formation dynamics behave nicely.  However, the set $\mathcal X_0$ itself may be difficult to characterize, and Theorem \ref{thm_orbiterrordynamics} does not guarantee, for example, convergence to the correct symmetric formation shape.  Analogous to classic results from formation control (see \cite{Sun2016}), by imposing additional assumptions on the framework we can show that in fact the error dynamics locally converge exponentially fast to the origin. 

\begin{theorem}\label{thm.orbitsym_main}
    Let $p^\star$ be the target formation satisfying conditions (i) and (ii) of Problem \ref{problem_symaquisition}, and assume that $(\G, p^\star)$ is a $\tau(\Gamma)$-symmetric isostatic framework satisfying Assumption \ref{ass.free}.  Furthermore, assume that Assumption \ref{vertexorbit_subgraph} holds and that the matrices $E(\Gamma_i)$ are constructed using only a spanning tree subgraph of $\G(\Gamma_i)$.  Then the origin is a locally exponentially stable equilibrium of \eqref{error_dynamics_orbit}.
\end{theorem}
%
\begin{proof}
    Let $\Psi(\rho) = \{\bar e \,:\, V(\bar e) \leq \rho\}$ for a sufficiently small $\rho$ such that for all points in $\Psi(\rho)$ the formation is $\tau(\Gamma)$-symmetric isostatic.  By Corollary \ref{cor.orbitisostatic}, the orbit rigidity matrix must have full row-rank on this set.  We now consider the set $\mathcal X_0$ in \eqref{orbit_Error_eq}.  Note that by Assumption \ref{vertexorbit_subgraph}, the matrices $E(\Gamma_i)$ have full column rank, and in particular, so must $\bar E_f(\Gamma)$. Therefore, $\bar q = 0$.  Similarly, since $\mathcal O^T$ also has full column rank in $\Psi(\rho)$, it follows that $\bar \sigma = 0$ and we conclude that $\mathcal X_0 = \{0\}$.

   We conclude the proof using the same Lyapunov function from the proof of Theorem \ref{thm_orbiterrordynamics}. Observe that due to the assumptions of the corollary, $\lambda_{min}  = \min_{e \in \Psi(\rho)} \lambda(\mathcal{M})>0$.  Since $\Psi(\rho)$ is compact, the existence of $\lambda_{min}$ is guaranteed.  Then one has
    $$\dot V(\bar e(t)) \leq -\lambda_{min} \|\bar e(t)\|^2 = - 2\lambda_{min} V(\bar e(t)),$$
which indicates that $V(e(t))$ is negative definite for $e(t)\in \Psi(\rho) \setminus \{0\}$. Thus the exponential stability of the equilibrium $e = 0$ in the error system \eqref{error_dynamics_orbit} is proved.
\end{proof}

{\color{black}{
In summary, minimally forced-symmetric rigidity provides an architecture for Problem \ref{problem_symaquisition} that ensures (local) exponential stability to the desired symmetric formation shape.  As in rigidity-based formation control strategies, this result is 
local since we must still be concerned with flip ambiguities of the framework.
We also note that the control \eqref{orbit_symformation} requires
fewer edges than the ordinary rigidity-based control.
\begin{theorem}\label{thm.number_of_edges}
In the setting of Theorem~\ref{thm.orbitsym_main}, 
the control \eqref{orbit_symformation} uses at most $(1+1/|\Gamma|)|{\cal V}|$ edges.
\end{theorem}
\begin{proof}
Let ${\cal G}_0=({\cal V}_0, {\cal E}_0)$ be the quotient graph as used in the description of the control.
By the assumption that the vertex set of each orbit induces a tree in the communication graph,
the control uses 
$|{\cal E}_0|+(|\Gamma|-1)|{\cal V}_0|$ edges.
It is known that, if $({\cal G},p^{\ast})$ is a $\tau(\Gamma)$-symmetric isostatic framework, then $|{\cal E}_0|\leq 2|{\cal V}_0|$, see, e.g., \cite[Theorem 62.1.4]{Bernd2017sym}.
This relation is the symmetric analogue of the well known fact that 
any isostatic framework has $2|{\cal V}|-3$ edges. Hence, the number of edges in the control is bounded by $(|\Gamma|+1)|{\cal V}_0|$. 
By the assumption that $\Gamma$ acts freely on ${\cal V}$, every vertex orbit is of size $|\Gamma|$, and hence  $|{\cal V}|=|\Gamma||{\cal V}_0|$. Thus,
 $(|\Gamma|+1)|{\cal V}_0|=(1+1/|\Gamma|)|{\cal V}|$.
\end{proof}
The bound in Theorem~\ref{thm.number_of_edges}  is significantly smaller than that required for stabilizing infinitesimally rigid frameworks without any additional symmetry constraints (which requires at least $2|\V|-3$ edges).

It should also be noted that, 
if a $\tau(\Gamma)$-symmetric framework $({\cal G},p^{\ast})$ is infinitesimally rigid in the ordinary sense (that is, the framework has a rigidity matrix with rank $2|{\cal V}|-3$), 
then one can always extract a subgraph ${\cal H}$
such that $({\cal H},p^{\ast})$ is $\tau(\Gamma)$-symmetric isostatic.
Hence, as long as Assumption~\ref{vertexorbit_subgraph} holds,
our control can be applied to a subgraph of ${\cal G}$ with $|{\cal E}_0|+(|\Gamma|-1)|{\cal V}_0|$ edges.

\begin{example}\label{ex_thm3}
We now consider the graph on $n=10$ nodes shown in Figure \ref{fig.rotsym1}.  This is a $\tau(\Gamma)$-symmetric framework with $\Gamma$ corresponding to the rotational group, where $\tau(\gamma)$ is the rotation matrix of $2\pi/5$ radians.  Note that this graph has $15$ edges.  To solve the formation control problem using the standard approach in \eqref{formation_control}, we would require an additional $2$ edges to ensure minimal infinitesimal rigidity of the framework (17 total).  On the other hand, the control strategy in \eqref{orbit_symformation} requires only $9$ edges, with the subgraph shown in Figure \ref{fig.rotsym2}.  The quotient graph for this framework is shown in Figure \ref{fig.rotsym3}.  Figure \ref{fig.rotsym4} shows the resulting trajectories when running \eqref{orbit_symformation}.
\end{example}
\vspace{-.5cm}
\textcolor{black}{
\begin{remark}
    One well-known challenge in formation controllers of the form \eqref{formation_control} is to avoid the invariant sub-space of co-linear solutions \cite{Krick2009}.  While the symmetry-forced formation control \eqref{orbit_symformation} is not immune to this problem, it depends greatly on the chosen symmetry.  For example, rotational symmetries, as in Example \ref{ex_thm3}, do not have co-linear solutions as an invariant set, while a reflection symmetry only has co-linear invariant solutions that are orthogonal to the mirror axis.
\end{remark}
}
\begin{figure*}[!t]
\centering
\subfigure[A $\tau(\Gamma)$-symmetric isostatic framework.]
	{       \begin{tikzpicture}
           0.6180   -1.9021
 
            \Vertex[size=.2,x=2.0, y=0, label=$1$, position=right,opacity =0]{v1} 
            \Vertex[size=.2,x=0.618, y=1.9021, label=$2$, position=above,opacity =0]{v2}
            \Vertex[size=.2,x=-1.618, y=1.1756, label=$3$, position=left,opacity =0]{v3}
            \Vertex[size=.2,x=-1.618, y=-1.1756, label=$4$, position=below,opacity =0]{v4}
            \Vertex[size=.2,x=.618, y=-1.9021, label=$5$, position=right,opacity =0]{v5}
            \Vertex[size=.2,x=1.0, y=0, label=$6$, position=right,opacity =0]{v6} 
            \Vertex[size=.2,x=0.3090, y=.9511, label=$7$, position=below,opacity =0]{v7}
            \Vertex[size=.2,x=-.809, y=.5878, label=$8$, position=left,opacity =0]{v8}
            \Vertex[size=.2,x=-.809, y=-.5878, label=$9$, position=left,opacity =0]{v9}
            \Vertex[size=.2,x=.309, y=-.9511, label=$10$, position=below,opacity =0]{v10}

            \Edge(v1)(v2)
            \Edge(v2)(v3)
            \Edge(v3)(v4)
            \Edge(v4)(v5)
            \Edge(v5)(v1)
            \Edge(v1)(v7)
            \Edge(v2)(v8)
            \Edge(v3)(v9)
            \Edge(v4)(v10)
            \Edge(v5)(v6)
            \Edge(v6)(v7)
            \Edge(v7)(v8)
            \Edge(v8)(v9)
            \Edge(v9)(v10)
            \Edge(v10)(v6)

       \end{tikzpicture}\label{fig.rotsym1}}
 \subfigure[Subgraph used to implement control strategy.]
	{       \begin{tikzpicture}
           0.6180   -1.9021
 
            \Vertex[size=.2,x=2.0, y=0, label=$1$, position=right,opacity =.5, color=blue]{v1} 
            \Vertex[size=.2,x=0.618, y=1.9021, label=$2$, position=above,opacity =.5, color=blue]{v2}
            \Vertex[size=.2,x=-1.618, y=1.1756, label=$3$, position=left,opacity =.5, color=blue]{v3}
            \Vertex[size=.2,x=-1.618, y=-1.1756, label=$4$, position=below,opacity =.5, color=blue]{v4}
            \Vertex[size=.2,x=.618, y=-1.9021, label=$5$, position=right,opacity =.5, color=blue]{v5}
            \Vertex[size=.2,x=1.0, y=0, label=$6$, position=right,opacity =.5, color=green]{v6} 
            \Vertex[size=.2,x=0.3090, y=.9511, label=$7$, position=below,opacity =.5, color=green]{v7}
            \Vertex[size=.2,x=-.809, y=.5878, label=$8$, position=left,opacity =.5, color=green]{v8}
            \Vertex[size=.2,x=-.809, y=-.5878, label=$9$, position=left,opacity =.5, color=green]{v9}
            \Vertex[size=.2,x=.309, y=-.9511, label=$10$, position=below,opacity =.5, color=green]{v10}

            \Edge[color=red](v1)(v2)
            \Edge(v2)(v3)
            \Edge(v3)(v4)
            \Edge(v4)(v5)
            \Edge[color=red](v1)(v7)
            \Edge[color=red](v6)(v7)
            \Edge(v7)(v8)
            \Edge(v8)(v9)
            \Edge(v9)(v10)

       \end{tikzpicture}\label{fig.rotsym2}}
 \subfigure[Quotient graph.]
   {       \begin{tikzpicture}

            \Vertex[size=.2,x=-.75, y=-.75, label=$1$, position=right,opacity =.5, color=blue]{v1} 
            \Vertex[size=.2,x=.75, y=.75, label=$6$, position=right,opacity =.5, color=green]{v6}
           
            \Edge[Direct, label={$\gamma$}, position=above, fontsize=\large, position={below=1mm}](v1)(v6)
            \Edge[Direct,loopposition=180,loopsize=1.5cm, label={$\gamma$}, position=above, fontsize=\large, position=left](v1)(v1)
            \Edge[Direct,loopposition=150,loopsize=1.5cm, label={$\gamma$}, position=above, fontsize=\large, position=left](v6)(v6)

       \end{tikzpicture}\label{fig.rotsym3}}
 \subfigure[Trajectories generated from \eqref{orbit_symformation}.]
{\label{fig:ex_orbit1}\includegraphics[width=0.45\columnwidth]{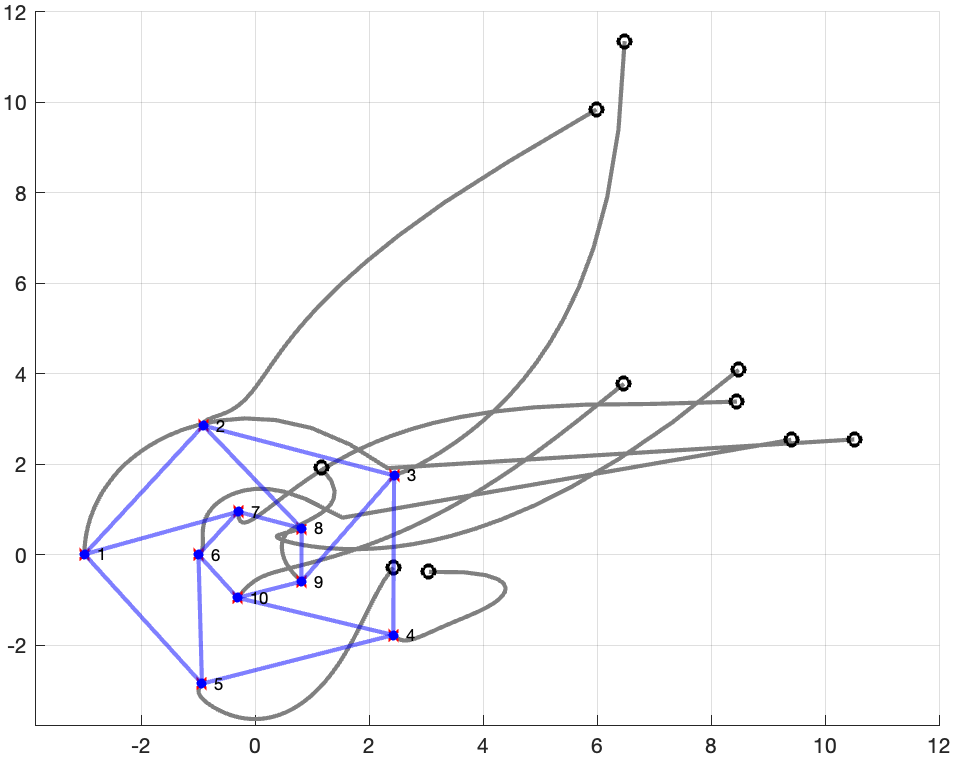}\label{fig.rotsym4}}   \subfigure[Trajectories with centroid consensus term, \eqref{orbit_symformation_consensus}, \eqref{consensus}.]
{\label{fig:ex_orbit2}\includegraphics[width=0.45\columnwidth]{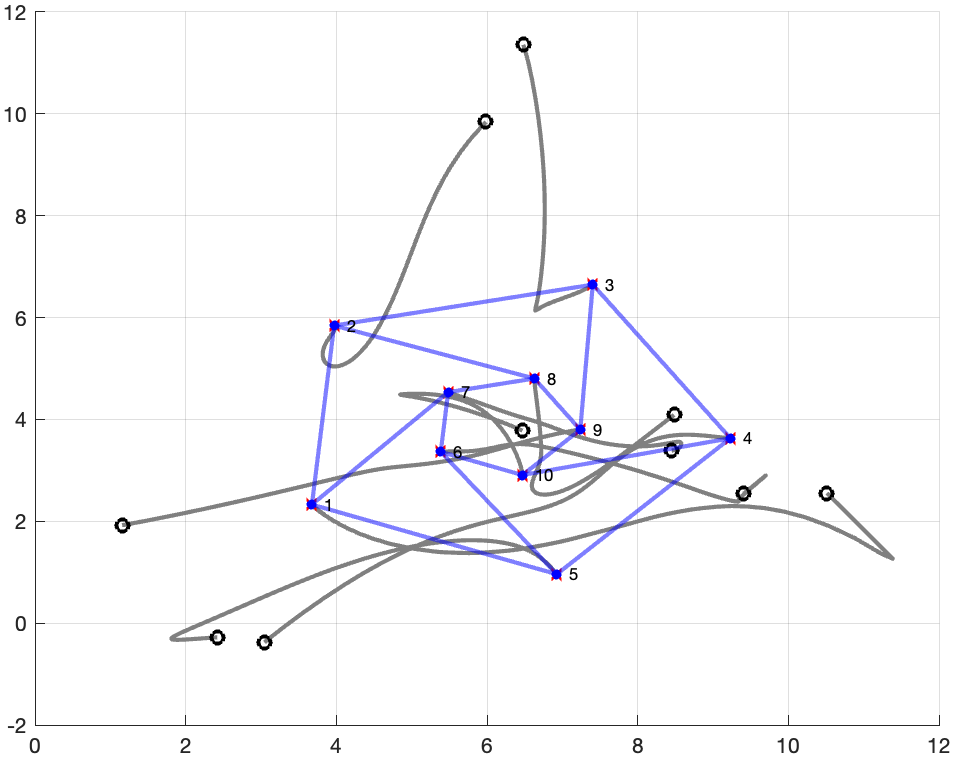}
  }
		\caption{Results from Example \ref{ex_thm3}.  We consider the $\tau(\Gamma)$-symmetric framework in (a) with the $2\pi/5$ rotational symmetry.  The framework has two vertex orbits, indicated by the blue and green nodes in (b), along with three edges in the edge orbit (red edges in (b)).  For the implementation of \eqref{orbit_symformation} we only require a spanning tree subgraph from each $\G(\Gamma_i)$ and the edges in $\E_0$, shown in (b).  Figure (c) shows the corresponding quotient gain graph.  Figures (d) and (e) show the resulting trajectories without and with the additional consensus term, respectively.}\label{fig:specialsym}
\end{figure*}
Of note is that  the symmetries used to define the $\tau(\Gamma)$-symmetric framework are defined with respect to a common inertial frame (see Fig. \ref{fig.rotsym4}).  In the sequel we propose a modification to \eqref{orbit_Error_eq} to relax this point.

\subsection{Symmetry forced formations with centroid consensus}
The $\tau(\Gamma)$-symmetric frameworks by definition have the point-group symmetries defined with respect to some fixed inertial point (the origin).  As seen in Figure \ref{fig:ex_orbit1}, for agent initial conditions that are far from the origin, the control strategy in \eqref{orbit_symformation} will drive the agent to the correct formation, but with respect to the origin.  In fact, it would be more desirable for the agents to be able to arrange themselves to the correct symmetric formation with respect to \emph{any} point.  This is further a necessary requirement if we want to include formation maneuvering as well.

In this direction, we propose the addition of a consensus term that will allow the agents to distributedly agree on a different origin that is a function of the initial conditions. Each agent will augment its state-vector with the centroid estimate, $r_i(t) \in \mathbb{R}^d$.  The classic consensus algorithm is then run on these virtual states,
\begin{align}\label{consensus}
\hspace{-6pt}\dot r_i(t) = \sum_{ij \in \E} (r_j(t)-r_i(t)) \Leftrightarrow \dot r(t) = -(L(\G) \otimes I_d)r(t),\end{align}
where $L(\G)$ is the combinatorial graph Laplacian matrix of $\G$.  It is well-known that under the assumption that $\G$ is connected, we have \cite{Mesbahi2010} 
\begin{equation}\label{consensus_centroid}r(t) \to \mathbf{r}^\star =\mathbbm{1}_{|\V|} \otimes \left((1/|\V|) \mathbbm (\mathbbm 1^T\otimes I_d)r(0)\right).\end{equation}  We now are able to use the virtual state $r_i(t)$ for each agent to effectively shift the origin in the control \eqref{orbit_symformation}.

We  define the shifted state $\bar c(t) = \begin{bmatrix} c_0^T(t) & c_f^T(t) \end{bmatrix}^T=P(p(t)-r(t))$.  The modified control, together with \eqref{consensus}, then takes the form 
\begin{align}\label{orbit_symformation_consensus}
    \begin{bmatrix} \dot{p}_0(t) \\ \dot{p}_f(t) \end{bmatrix} &= \begin{bmatrix} -\mathcal{O}^T(\mathcal{G}_0,c_0(t))\bigg(\mathcal{O}(\mathcal{G}_0,c_0(t))c_0(t) - \mathrm{\bf d}_0^2 \bigg) \\0\end{bmatrix} \nonumber\\&+ PQP^T\begin{bmatrix}c_0(t) \\ c_f(t)\end{bmatrix}.
\end{align}
The combined dynamics \eqref{consensus} and \eqref{orbit_symformation_consensus} are in a cascade form. 
Following the same approach as before, we define the error signals
$$\bar{\sigma}(t) = \mathcal{O}(\mathcal{G}_0,c_0(t))c_0(t) - \mathrm{\bf d}_0^2\text{, and }\bar q(t)=\bar E(\Gamma)^TP^T\bar c(t).$$
The error dynamics can now be expressed as
{\footnotesize \begin{align}\label{consensus_error}
\begin{bmatrix} \dot{\bar \sigma }(t) \\ \dot{\bar q}(t) \end{bmatrix} &=-\begin{bmatrix} \mathcal{O}(\G_0, c_0(t))\mathcal{O}^T(\G_0,c_0(t)) & \mathcal{O}(\G_0,c_0(t)) \bar E_0(\Gamma) \\ \bar E_0^T(\Gamma)\mathcal{O}^T(\G_0,c_0(t)) & \bar E^T(\Gamma)\bar E(\Gamma) \end{bmatrix} \begin{bmatrix}  {\bar \sigma }(t) \\ {\bar q}(t) \end{bmatrix}\nonumber \\
& + \begin{bmatrix} \begin{bmatrix} \mathcal O(\mathcal G_0,c_0(t)) & 0 \end{bmatrix}P(L(\G) \otimes I_d) \\ \bar E(\Gamma)(L(\G) \otimes I_d) \end{bmatrix}  r(t)
\end{align}}

To assist in our analysis of the cascade system \eqref{consensus_error} and \eqref{consensus}, we perform a simple change of coordinates.  Let $\hat p(t) = p(t) - {\mathbf r}^\star$, $\hat r(t) = r(t) - {\mathbf r}^\star$. With this definition, note that $\hat c(t) = \hat p(t)-\hat r(t) = c(t)$, similarly for the error signals $\bar \sigma(t)$ and $\bar q(t)$ (i.e., $\hat \sigma (t) = \bar \sigma(t)$ and $\hat q(t) = \bar q (t)$).  With this definition, the shifted cascade system can be represented as
{\footnotesize{\begin{align}
    \begin{bmatrix} \dot{\hat \sigma }(t) \\ \dot{\hat q}(t) \end{bmatrix} &=-\begin{bmatrix} \mathcal{O}(\G_0, \hat c_0(t))\mathcal{O}^T(\G_0,\hat c_0(t)) & \mathcal{O}(\G_0,\hat c_0(t)) \bar E_0(\Gamma) \\ \bar E_0^T(\Gamma)\mathcal{O}^T(\G_0,\hat c_0(t)) & \bar E^T(\Gamma)\bar E(\Gamma) \end{bmatrix} \begin{bmatrix}  {\hat \sigma }(t) \\ {\hat q}(t) \end{bmatrix}\nonumber \\
& + \begin{bmatrix} \begin{bmatrix} \mathcal O(\mathcal G_0,\hat c_0(t)) & 0 \end{bmatrix}P(L(\G) \otimes I_d) \\ \bar E(\Gamma)(L(\G) \otimes I_d) \end{bmatrix}  \hat r(t) \label{shiftend_consensus_error1} \\
\dot{\hat r}(t) &= -(L(\G)\otimes I_d)\hat r(t). \label{shiftend_consensus_error2}
\end{align}}}

We now will show that the cascade system \eqref{shiftend_consensus_error1} and \eqref{shiftend_consensus_error2} is locally exponentially stable.  To do so, we must first show that these dynamics are locally input-to-state stable \cite{Sontag1996}.
\begin{theorem}\label{prop.cascadeLISS}
 Let $p^\star$ be the target formation such that $p^\star $ satisfies conditions (i) and (ii) of Problem \ref{problem_symaquisition}, and assume that $(\G, p^\star)$ is a $\tau(\Gamma)$-symmetric isostatic framework satisfying Assumption \ref{ass.free}. Assume that Assumption \ref{vertexorbit_subgraph} holds and that the matrices $E(\Gamma_i)$ are constructed using only a spanning tree subgraph of $\G(\Gamma_i)$.  Then the cascade system \eqref{shiftend_consensus_error1} and \eqref{shiftend_consensus_error2} is locally exponentially stable.  

\end{theorem}


\begin{proof}
First we establish that \eqref{shiftend_consensus_error1} is locally input-to-state stable  while interpreting the signal $\hat r(t)$ as the input.
    When $\hat r(t) = 0$ we have that $\hat c_0(t) = \hat  p_0(t)$ and the dynamics \eqref{shiftend_consensus_error1} reduce to \eqref{error_dynamics_orbit}.  Using the same arguments as in Theorem \ref{thm.orbitsym_main}, this system is locally exponentially stable, which shows the system must be locally input-to-state stable \cite{Sontag1996}.  To establish that the cascade system is exponentially stable, all we need is to recall that \eqref{shiftend_consensus_error2} is exponentially stable and converges from any initial condition to the origin \cite{Mesbahi2010}. It can then be concluded that the cascade system is also locally exponentially stable \cite{Sundarapandian2004}.      
\end{proof}

Finally, we can establish the agent trajectories will then converge to a symmetric configuration with respect to the centroid of the virtual state $r(t)$, defined in \eqref{consensus_centroid}.

\begin{theorem}\label{thm.consensus_orbit}
Consider the cascade dynamics \eqref{orbit_symformation_consensus} and \eqref{consensus} and assume that $(\mathcal G,p^\star - {\mathbf r}^\star)$
is a $\tau(\Gamma)$-symmetric isostatic framework.  Then for all initial conditions sufficiently close to $p^\star$, and ${\mathbf r}^\star$ given in \eqref{consensus_centroid}, $p(t)$ exponentially converges to $p^\star$.


\end{theorem}
\begin{proof}
    From Theorem \ref{prop.cascadeLISS} we have that $\hat \sigma(t)$ and $\hat q(t)$ converge to the origin.  The remainder of the proof follows the same argument as Theorem \ref{thm_orbiterrordynamics} but on the shifted state $\hat p(t)$.  Therefore, since $\hat p(t)$ locally and exponentially converges to a $\tau(\Gamma)$-symmetric isostatic framework, $\hat p(t) \to p^\star - {\mathbf r}^\star$, it follows that $p(t) \to p^\star + \mathbf{r}^\star$.\end{proof}
    
    


This result assumes that $p(0)-{\mathbf r}^\star$ is sufficiently close to a $\tau(\Gamma)$-symmetric configuration.  This may be problematic since in general the vector $\mathbf r^\star$ may not be globally known to all agents.  On the other hand, it is only as restrictive as Theorem \ref{thm_orbiterrordynamics} and so does not introduce any new conservativeness.   

\begin{continueexample}{ex_thm3}
    We return to the same setup as Example \ref{ex_thm3} and now implement the centroid consensus version of the dynamics \eqref{orbit_symformation_consensus}.  The resulting trajectories can be seen in Figure \ref{fig:ex_orbit2}. Here we initialize the virtual state $r(t)$ to be the initial condition of the agents, $r(0)=p(0)$. The formation converges to the correct symmetry-forced configuration, but now with respect to the centroid of the initial conditions.
\end{continueexample}

\section{Concluding Remarks}

This work presented a formation control strategy for controlling a network of agents to a formation characterized by its symmetry properties.  Leveraging recent results from symmetry-forced rigidity theory, a gradient control strategy was derived with two main components based on a chosen symmetry group of the formation: one maintains distances between edges in the edge orbits, while the other forces agents within the same vertex orbit to a symmetric position.  The stability results turn out to be related to properties of the so-called orbit rigidity matrix which is a central object in symmetry-force rigidity theory.  While the results in this work focused only on symmetries that are free, extending these ideas to the non-free case should be straightforward and is the topic of future research.  
\bibliographystyle{IEEEtran}
\bibliography{main}
\vspace{-1.5cm}
 \begin{IEEEbiography}[{\includegraphics[width=1in,height=1.25in,clip,keepaspectratio]{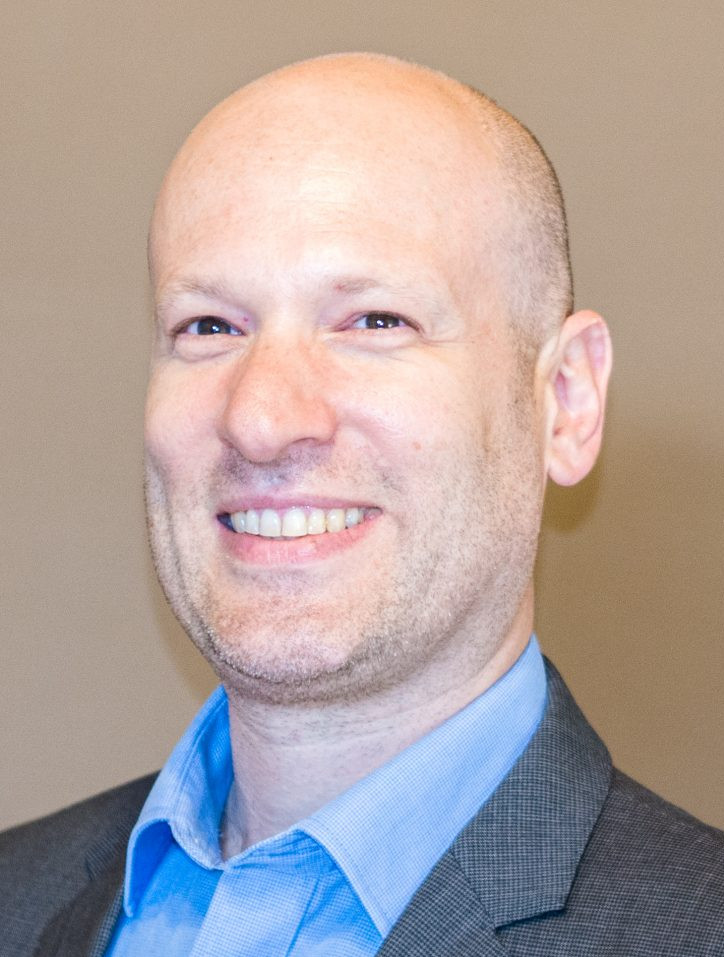}}]
 {\bf Daniel~Zelazo} (Senior Member, IEEE) received the B.Sc. and M.Eng. degrees in electrical engineering and computer science from the Massachusetts Institute of Technology, Cambridge, MA, USA, in 1999 and 2001, respectively, and the Ph.D. degree in aeronautics and astronautics from the University of Washington, Seattle, WA, USA, in 2009. From 2010 to 2012, he was a Postdoctoral Research Associate and Lecturer with the Institute for Systems Theory and Automatic Control, University of Stuttgart, Germany. He is an Associate Professor of aerospace engineering with the Technion-Israel Institute of Technology, Haifa, Israel. His research interests include topics related to multiagent systems.
 \end{IEEEbiography}
 \vspace{-1.5cm}
 \begin{IEEEbiography}[{\includegraphics[width=1in,height=1.25in,clip,keepaspectratio]{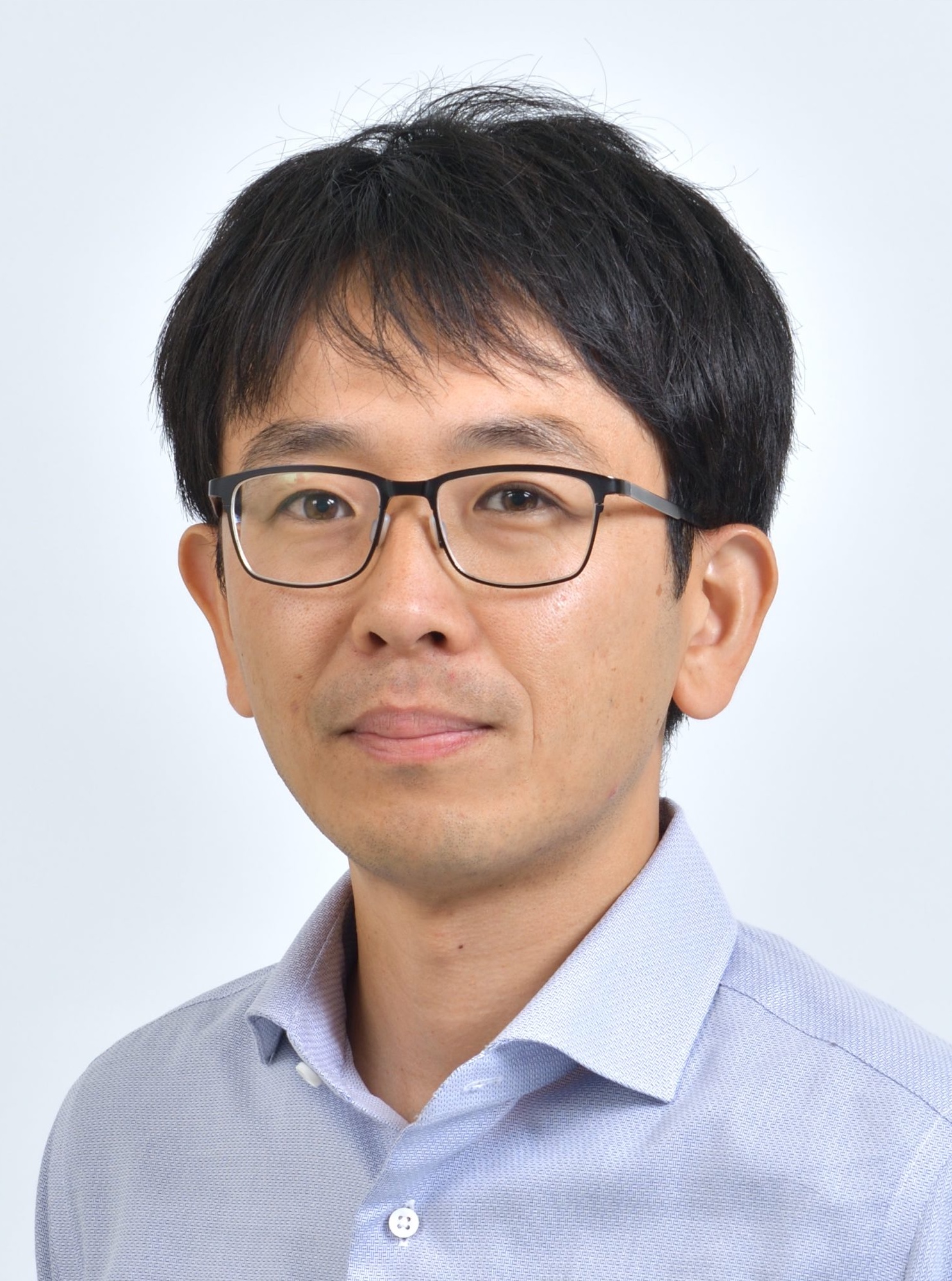}}]
 {\bf Shin-Ichi Tanigawa} received the Master of Engineering and Doctor of Engineering degrees from Kyoto University, Japan, in 2007 and 2010, respectively. He was a JSPS postdoctoral fellow at Kyoto University, Japan, from 2010 to 2011, and at CWI, Netherlands, from 2015 to 2017. From 2011 to 2017, he was an Assistant Professor in Mathematical Science at Kyoto University, Japan. He is currently an Associate Professor of Mathematical Informatics at the University of Tokyo, Japan. His research interests include discrete and computational geometry, graph theory, and combinatorial optimization.
 \end{IEEEbiography}
 \vspace{-1.5cm}
 \begin{IEEEbiography}[{\includegraphics[width=1in,height=1.25in,clip,keepaspectratio]{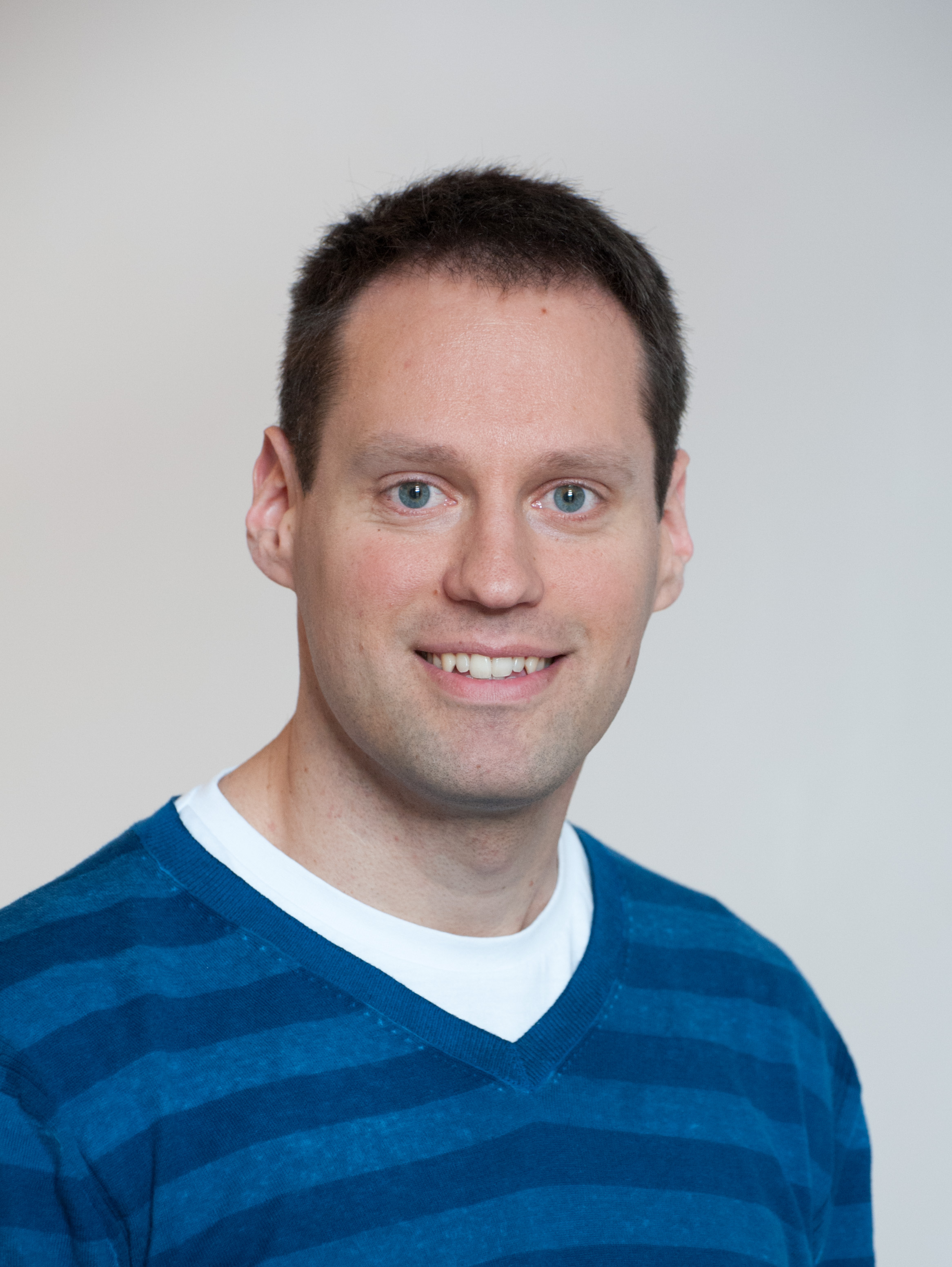}}]
 {\bf Bernd Schulze}  is a Reader in Mathematics at Lancaster University, UK. Before arriving at Lancaster in 2012 he held postdoctoral positions at the TU Berlin, Germany, from 2009 to 2011 and at the Fields Institute in Toronto, Canada, in 2011. Prior to that he studied mathematics and computer science at the FU Berlin, Germany, and at Western Michigan University, USA, and he received his Ph.D. in mathematics from York University, Canada, in 2009. His main research interests lie in applied discrete geometry, graph theory, combinatorial optimization, symmetry, and algebraic methods in discrete mathematics. 
 \end{IEEEbiography}

\end{document}